\documentclass[oneside]{article}
\usepackage[intlimits,reqno]{amsmath}
\usepackage{amsthm, graphics}
\usepackage{amssymb}
\usepackage{verbatim}
\include{diagxy}
\input{xy}
\xyoption{all}

\usepackage{enumerate}
\usepackage{amsmath}
\usepackage{verbatim}
\usepackage{bbm}
\newcommand{\norm}[1]{\left\Vert#1\right\Vert}

\newcommand{\No}{\mathbb{N}\cup\{0\}}

\newcommand{\N}{\mathbb N}

\newcommand{\A}{\mathcal{A}}
\newcommand{\B}{\mathcal{B}}

\newcommand{\J}{\mathcal{J}}

\renewcommand{\P}{\mathcal{P}}
\newcommand{\G}{\mathcal{G}}
\newcommand{\AG}{\mathcal{AG}}
\newcommand{\AU}{\mathcal{AU}}
\newcommand{\U}{\mathcal{U}}

\renewcommand{\H}{\mathcal{H}}

\newcommand{\Ss}{\textbf{s}}
\newcommand{\Tt}{\textbf{t}}

\newcommand{\M}{\mathfrak{M}}

\renewcommand{\to}{\rightarrow}

\newcommand{\too}{\rightrightarrows}
\newcommand{\prf}[1]{\begin{proof}#1\end{proof}}

\newcommand{\ol}{\overline}

\newcommand{\be}{\begin{equation}}
\newcommand{\ee}{\end{equation}}
\newcommand{\bse}{\begin{subequations}}
\newcommand{\ese}{\end{subequations}}
\newcommand{\ben}{\begin{enumerate}}
\newcommand{\een}{\end{enumerate}}
\newcommand{\bit}{\begin{itemize}}
\newcommand{\eit}{\end{itemize}}
\newcommand{\bex}{\begin{example}}
\newcommand{\eex}{\begin{flushright}$\diamondsuit$\end{flushright}\end{example}}

\newtheorem{thm}{Theorem}[section]
\newtheorem{cor}[thm]{Corrolary}
\newtheorem{lem}[thm]{Lemma}

\newtheorem{prop}[thm]{Proposition}

\theoremstyle{definition}

\newtheorem{example}{Example}
\numberwithin{example}{section}
\theoremstyle{remark}
\newtheorem{rem}[thm]{Remark}

\numberwithin{equation}{section}

\renewcommand{\emph}[1]{{\bf #1}}

\begin{document}
\title{Groupoids and inverse semigroups associated to $W^*$-algebras}
\author{Anatol Odzijewicz,\ Aneta Sli\.{z}ewska\\Institute of Mathematics\\
University in Bia{\l}ystok
\\Akademicka 2, 15-267 Bia{\l}ystok, Poland}

\maketitle
%\shorttoc{Contents}{1}
\tableofcontents
\vspace{1cm}
\begin{abstract}
We investigate the Banach Lie groupoids and inverse semigroups naturally associated
to $W^*$-algebras. We also present statements describing the relationship between
these groupoids and the Banach Poisson geometry which follows in the canonical way from the $W^*$-algebra structure.

\end{abstract}

\section{Introduction}

During recent decades
the notion of groupoid entered many branches of mathematics including topology  \cite{brown}, differential geometry in general \cite{mac2} and  Poisson geometry \cite{Wei} in particular  as well as the theory of operator algebras \cite{pater}. Let us recall shortly that  a groupoid
 is a small category all of whose morphisms are invertible.
In accordance with \cite{mac2} they are ``the natural formulation of a symmetry for objects which have bundle structure''.
Nevertheless  the role of groupoids is not so widely accepted as
that of groups.  On the other hand the theory of $W^*$-algebras (von Neumann algebras) occupies an outstanding place in  mathematics and mathematical physics \cite{sakai}.

 Motivated by the existence of the canonically   defined Banach Lie-Poisson structure on the predual $\M_*$ of any $W^*$-algebra $\M$, see \cite{OR}, and by the importance of this structure in the theory of infinite dimensional Hamiltonian systems, see \cite{OR2}, we clarify here some natural connections between Banach Poisson geometry and groupoid theory from one side and $W^*$-algebras from the other.

In Section 2 we show that the  structure of any  $W^*$-algebra $\M$ naturally defines
two important groupoids $\U(\M)$ and $\G(\M)$
the first of which consists of the partial isometries in $\M$
and the second, being the complexification of $\U(\M)$, consists of
the partially invertible elements of $\M$. In this section we also discuss canonical actions of $\G(\M)$ and $\U(\M)$ on the lattice of projections $\mathcal{L}(\M)$ and on the cone $\M^+_*$ of the positive normal states on $\M$. In the Theorem \ref{rep} we show that the action groupoid $\U(\M)*\M_*^+$ has the faithful representation on the GNS bundle $\ol{\mathbb{E}}\to \M_*^+$. The Theorem \ref{propLambda} shows that one can consider  $\U(\M)*\M_*^+$ as a subgroupoid of the groupoid of partial isometries $\U(\ol\rho(\M)')$ for the commutant $\ol\rho(\M)'$ of the $W^*$-representation $\ol{\rho}:\M\to L^\infty(L^2\Gamma(\mathbb{E},\M^+_*))$ of $\M$ in the Hilbert space of the square-summable sections for the bundle $\ol{\mathbb{E}}\to \M_*^+$.

In Section 3 we study inverse semigroups which are the subsets of the groupoid $\U(\M)$
for various type of $W^*$-algebras.
Such inverse semigroups form a class of subgroupoids of $\U(\M)$. In particular we describe the inverse semigroups related to matrix unit inverse semigroups, see Theorem \ref{3.3}, the Clifford inverse semigroups (Proposition \ref{centr}), and the monogenic inverse semigroups  (Theorem \ref{monogthm}). Finally we discuss the CAR inverse semigroup given by the canonical anticommutation relations which has fundamental meaning  in quantum physics.

The topology and Banach manifold structure of $\G(\M)$ and $\U(\M)$ are described in Section 4. We show here that $\G(\M)$ is not a topological groupoid with respect to any natural topology of $\M$. However the groupoid $\U(\M)$
 is a topological groupoid with respect to the uniform topology as well as the $s^*(\U(\M),\M_*)$-topology. Theorem \ref{4.3} describes the topological structure of the action groupoids related to $\U(\M)$. We investigate the complex Banach manifold structure on the lattice $\mathcal{L}(\M)$ and the groupoid $\G(\M)$ and show that $\G(\M)$ is a Banach-Lie groupoid with $\mathcal{L}(\M)$ as its base manifold, see Theorem \ref{BLgroupoidthm}. The last statement is also true for the groupoid $\U(\M)$ when we consider it as a real Banach manifold, and the groupoid  $\G(\M)$ is the complexification of $\U(\M)$, see Theorem \ref{real}.

In Section 5 we present several, essential in the present context, statements describing relationship between groupoids $\G(\M)$ and $\U(\M)$  and the canonical  Poisson
structure defined on the Banach vector bundles  $\A_*\G(\M)$ and $\A_*\U(\M)$ predual to the algebroids $\AG(\M)$ and $\AU(\M)$ respectively. From that we conclude that in the framework  of $W^*$-algebras theory there exists the natural illustration    of the deep ideas connecting finite dimensional Poisson geometry and Lie groupoids theory which was investigated in  \cite{kara}, \cite{mac2}, \cite{weinstein2}.

\section {Groupoids associated to $W^*$-algebras and their representations}

In this section we introduce  various groupoids   defined in a canonical way by  the
given $W^*$-algebra $\M$. We will also describe representations of these groupoids on vector bundles
related to the algebra $\M$ as well as to its predual $\M_*$. The basic facts from the theory of groupoids
one can find in appendix to this
paper. The detailed presentation of the groupoid theory one can find in \cite{mac2}. The part of the theory of $W^*$-algebras indispensable  for the subsequent investigations is given in \cite{sakai} and \cite{takesaki}.

 By $\mathcal{U}(\M)$ we shall denote the set of all partial
    isometries in $\M$, i.e. $u\in \mathcal{U}(\M)$ if and only if $u^*u\in \mathcal{L}(\M)$, where $\mathcal{L}(\M)$
    is the lattice of  orthogonal projections $p=p^*=p^2\in \M$. For $x\in \M$ one defines the
    left support $l(x)\in \mathcal{L}(\M)$ (respectively  right support $r(x)\in \mathcal{L}(\M)$) as the smallest projection
    in $\M$ such that  $l(x)x=x$ (respectively $x\ r(x)=x$). If $x=x^*$ then $l(x)=r(x)=:s(x)$ and one
    calls $s(x)$ the support of $x$. Let \be\label{polar} x=u|x|\ee
be the polar decomposition of $x$, where $u\in\mathcal{ U}(\M)$ and
$|x|\in\M^+:=\{x\in\M:\ \ x^*=x>0\}$, see \cite{sakai} . Then one has
\be\label{support}\begin{array}{c} l(x)=s(|x^*|)=uu^*,\\
r(x)=s(|x|)=u^*u.\end{array}\ee

 In this paper we will
denote by ${G}(p\M p) $ the group of all invertible elements
 of the  $W^*$-subalgebra $p\M p\subset \M$. In particular if $p=1$ then $G(\M)$ will be the group
 of all invertible elements of $\M$. Similarly by ${U}(p\M p)$ and ${U}(\M)$
 we will denote the groups of unitary elements of $p \M p$ and $\M$.

 For any $x\in \M$ one has $|x|\in p\M p$, where $p=s(|x|)$. Let us define the subset $\G(\M)\subset\M$ by
 \be\mathcal{G}(\M):=\{x\in \M:\ \ |x|\in G(p \M p), \ {\rm where } \ p=s(|x|)\}.\ee
 Equivalently, $x$ belongs to $\ \G(\M)$ if the left multiplication $L_{|x|}$ by $|x|$ defines a right $\M$-module isomorphism
 \be\label{leftisom} L_{|x|}:p\M\ \ \tilde{\to}\ \  p\M\ee
 of the right ideal $p\M$.

\begin{prop}
The subset $\mathcal{G}(\M)\subset \M$ has the canonical  structure of a
groupoid with  $\mathcal{L}(\M)$ as the base set. The groupoid structure of $
\mathcal{G}(\M)$  is defined as follows: \ben[(i)]
\item the identity section $\varepsilon: \mathcal{L}(\M)\hookrightarrow \mathcal{G}(\M)$ is the inclusion;
\item the source and target maps: $\Ss,\Tt:\mathcal{G}(\M)\rightarrow \mathcal{L}(\M)$ are defined  by
\be\label{moment} \Ss(x):=r(x)=u^*u\ \ \  {\rm and} \ \ \  \Tt(x):=l(x)=uu^*;\ee
\item the product
\be\label{prod} \mathcal{G}(\M)^{(2)}\ \ni\ (x,y)\mapsto \ xy\ \in\mathcal{G}(\M)\
\ee on the set of composable pairs $$\mathcal{G}(\M)^{(2)}
:=\{(x,y)\in\mathcal{G}(\M)\times\mathcal{G}(\M);\ \ \Ss(x)=\Tt(y)\}$$  is given by
the product in the $W^*$-algebra $\M$;
\item the inverse map $\iota:\mathcal{G}(\M)\rightarrow \mathcal{G}(\M)$ is given by \be\label{inverse}\iota (x):=|x|^{-1}u^*, \ee
where $u$ and $|x|$ are defined in the unique way by the polar decomposition
(\ref{polar}). \een\end{prop}
 \prf{Since $\varepsilon:\mathcal{L}(\M)\rightarrow
\mathcal{G}(\M)$ is inclusion the maps $\Tt:\mathcal{G}(\M)\to \mathcal{L}(\M)$ and
$\Ss:\mathcal{G}(\M)\to \mathcal{L}(\M)$ are surjective.

From $x\Ss(x)=x$ one has $yx\Ss(x)=yx$. This gives $\Ss(yx)\leqslant \Ss(x)$,
where $" \leqslant "$ means lattice order in $\mathcal{L}(\M)$. Using
$r(y)=\iota(y)y=l(x)$ and $\iota(y)yx\ \Ss(yx)=\iota(y)yx$ we obtain $x\
\Ss(yx)=x$. Thus  we have $\Ss(x)\leqslant \Ss(yx)$. This shows that
$r(yx)=\Ss(yx)=\Ss(x)=r(x)$. In a similar way we show that $l(yx)=l(y)$.

The associativity of  the product (\ref{prod}) follows from the associativity of the
algebra  product.

  Using (\ref{polar}) and (\ref{support}) we get
\be\label{groupoid}\begin{array}{c}\iota (x)x=\varepsilon (\Ss(x)),\\
 x\iota(x)=\varepsilon(\Tt(x)),\\
\Ss(x)=\Tt(\iota(x)),\\
\Tt(x)=\Ss(\iota(x))\end{array}\ee
for $x\in \mathcal{G}(\M)$. The above proves the groupoid structure of $\G(\M)$.}

From now on we will call $\mathcal{G}(\M)$ the groupoid
 of partially invertible elements of the $W^*$-algebra $\M$.

Since \be\label{cos} |x^*|=u|x|u^*\ee we see  that the groupoid $\G(\M)$ is invariant with respect to $*$-involution. Thus from the definition of the inverse map $\iota:\G(\M)\to \G(\M)$ follows that the involution
$J:\G(\M)\to\G(\M)$ defined by
\be\label{J} J(x):=\iota(x)^*=\iota(x^*)\ee
is an automorphism of the groupoid $\G(\M)$. We note also that the set of fixed points of $J:\G(\M)\to\G(\M)$, i.e.
\be\label{parisom} \{u\in \G(\M):\quad J(u)=u\}\ee
is the set $\mathcal{U}(\M)$ of all partial isometries of the $W^*$-algebra $\M$. From the above  we conclude the following
\begin{prop}\label{cor1}
The set $\mathcal{U}(\M)$ of partial isometries in $\M$ is a wide subgroupoid
of the groupoid $\mathcal{G}(\M)$.
\end{prop}

The other important wide subgroupoid of $\G(\M)$ is the inner subgroupoid $\mathcal{J}(\M)\subset\G(\M)$ defined by
\be\label{isotropy}
\mathcal{J}(\M):=\bigcup_{p\in \mathcal{L}(\M)}(\Ss^{-1}(p)\cap \Tt^{-1}(p)). \ee
It is a totally intransitive subgroupoid and one can consider it as a bundle $s: \mathcal{J}(\M)\to \mathcal{L}(\M)$ of
groups $\Ss^{-1}(p)\cap \Tt^{-1}(p)=G(p\M p)$ indexed by $p\in \mathcal{L}(\M)$. One has also the action
$I:\mathcal{G}(\M)* \mathcal{J}(\M)\to \mathcal{J}(\M)$ of $\mathcal{G}(\M)$ on
$\mathcal{J}(\M)$  defined by \be\label{inner} I_x y:=xy\ \iota(x)\ee for $(x,y)\in
\mathcal{G}(\M)* \mathcal{J}(\M):=\{(x,y)\in\mathcal{G}(\M)\times
\mathcal{J}(\M): r(x)=s(y)\}$.  This action is called the inner action. Note that the moment map for the inner action  is the support map $s:\mathcal{J}(\M)\to\mathcal{L}(\M)$. Since for $y\in\mathcal{J}(\M)$ one has $s(y)=l(y)=r(y)$  one can consider the lattice of projections $\mathcal{L}(\M)$ as a subgroupoid  of $\J(\M)$ invariant  with
respect to the action $I:\mathcal{G}(\M)*\J(\M)\to \J(\M)$. So, the inner action
(\ref{inner}) defines the action $I:\mathcal{G}(\M) * \mathcal{L}(\M)\to \mathcal{L}(\M)$ of the
groupoid $\mathcal{G}(\M)$ on the lattice $\mathcal{L}(\M)$. The moment map for this
action is the identity map $id:\mathcal{L}(\M)\to \mathcal{L}(\M)$.

The groupoid structure of $\mathcal{G}(\M)$ allows us to define the principal bundles:
\be\label{bundles}\begin{array}{c} \Ss:\Tt^{-1}(p) \to \mathcal{O}_p\\
\Tt:\Ss^{-1}(p) \to \mathcal{O}_p\end{array}\ee over the orbit $\mathcal{O}_p:=\{xp\iota(x):\
x\in\Ss^{-1}(p)\}$ of the inner action $I:\mathcal{G}(\M)*\mathcal{L}(\M)\to \mathcal{L}(\M)$ of
$\mathcal{G}(\M)$ on the lattice $\mathcal{L}(\M)$. The structural group for the principal
bundles (\ref{bundles}) is the group $G(p\M p)$.

 The inner action (\ref{inner}) defines the equivalence relation on $\mathcal{L}(\M)$:
\be\label{eqiv rel1} p\thicksim q\ \ \ {\rm iff}\ \ \ q\in\mathcal{O}_p,\ee for which the
equivalence class $[p]$ is the orbit $\mathcal{O}_p$ generated from the projection $p\in
\mathcal{L}(\M)$. Let $\mathcal{L}(p):=\{q\in \mathcal{L} (\M):\ \ q\leqslant p\}\subset \mathcal{L}(\M)$ be the lattice
ideal of the subprojections of the projection $p\in \mathcal{L}(\M)$. One has the
canonically defined relation $\prec $ on the set of  equivalence classes of the equivalence relation (\ref{eqiv rel1}), i.e. \be\label{eqiv rel2}[q]\prec [p]\ \ \ {\rm
iff}\bigcup_{q'\in[q]} \mathcal{L}(q') \ \subset\ \bigcup_{p'\in[p]} \mathcal{L}(p').\ee

From  Proposition \ref{cor1} we conclude

\begin{cor}
The inner actions of $\mathcal{G}(\M)$ and  $\mathcal{U}(\M)$ on $\mathcal{L}(\M)$
have the same orbits.
\end{cor}

Therefore one can define the equivalence relation (\ref{eqiv rel1}) and the
relation (\ref{eqiv rel2}) taking $\mathcal{U}(\M)$ instead of $\mathcal{G}(\M)$.

\begin{prop}
The relation $\prec $  defined in (\ref{eqiv rel2}) is a order relation on the set of
the inner action orbits of groupoid $\G(\M)$ on $\mathcal{L}(\M)$. If  $\M$ is a factor then this order is linear.
\end{prop}
\prf{ Let us denote by $max[p]$ the set of maximal elements of the orbit
$\mathcal{O}_p$. Since $\mathcal{L}(p)\subset \mathcal{L}(q)$ iff $p\leqslant q$ we have $\bigcup_{p'\in
max[p]} \mathcal{L}(p') =\bigcup_{p'\in [p]} \mathcal{L}(p')$. Thus $[p]\prec [q]$ and $[q]\prec [p]$
iff $\bigcup_{p'\in [p]} \mathcal{L}(p')=\bigcup_{q'\in[q]} \mathcal{L}(q').$ From the above it follows
that $p'\in \mathcal{L}(q')$ for some $q'\in max[q]$. Since $p'\in[p]$ and $q'\in[q]$ are
maximal elements of the orbits we obtain $[p]=[q]$, i.e. the relation $\prec$ is
antisymmetric.
The proof of reflexivity and transitivity of the relation  $\prec$ is trivial.\\
In the factor case the linearity of order relation $\prec $  follows from Comparability Theorem,
e.g. see \cite{sakai}, \cite{takesaki}.}

The equivalence relation (\ref{eqiv rel1}) is fundamental  for the
classification of $W^*$-algebras, see \cite{sakai}, \cite{takesaki}. So, the
problem of classification of $\mathcal{U}(\M)$-orbits on $\mathcal{L}(\M)$ is strictly
related to the Murray and von Neumann classification of $W^*$-algebras. The reason is that the inner
action $I:\mathcal{U}(\M) * \mathcal{L}(\M)\to \mathcal{L}(\M)$ preserves the lattice structure
of $\mathcal{L}(\M)$, i.e. for any $(u,p)\in \mathcal{U}(\M)*\mathcal{L}(\M)$ the maps \be I_u:\
\mathcal{L}(p)\ \to\ \mathcal{L}(upu^*)\ee are isomorphisms of the lattice ideals. In particular if the
projection $z\in \mathcal{L}(\M)\cap \mathcal{Z}(\M)$ is central, where $\mathcal{Z}(\M)$
is the center of $\M$, then the lattice $\mathcal{L}(z)=\mathcal{L}(z\M)$ is preserved by the
inner action.  This allows us to reduce the classification of
$\mathcal{U}(\M)$-orbits on $\mathcal{L}(\M)$ to the classification of
$\mathcal{U}(\M)$-orbits for the case when $\M$ is a factor.

In the subsequent part of this section we investigate the representations of the groupoids on the vector bundles which are given by the structure of the $W^*$-algebra.

Let us begin by briefly explaining that what one understands by representation of the groupoid is a direct generalization of the notion of group representation in the vector space. However for groupoids one takes a vector bundle instead of the vector space. For the purposes of this paper as a rule we assume that the fibres $\pi^{-1}(m)$, $m\in M$, of vector bundle $(\mathbb{E},M, \pi:\mathbb{E}\to M)$ under consideration will be not necessary isomorphic. In consequence of that the structural groupoid $\G(\mathbb{E})$ of this  bundle  would be not necessary transitive on base $M$.

Recall, see also \cite{mac2}, that the structural groupoid $\G(\mathbb{E})$ consists of linear isomorphisms $e_m^n:\mathbb{E}_m\ \tilde{\to}\ \mathbb{E}_n$ between the fibers of $\pi: \mathbb{E}\to M$. The base of $\G(\mathbb{E})$ is the base set $M$ of the bundle. The source map $\Ss:\G(\mathbb{E})\to M$ and the target map $\Tt:\G(\mathbb{E})\to M$ are defined by $\Ss(e_m^n):=m$ and  $ \Tt(e_m^n):=n$ respectively. The inverse map is given by $\iota(e_m^n):=(e_m^n)^{-1}$ and the identity section by $\varepsilon(m):=id_m^m$. Finally the product of isomorphisms  $e_l^m:\mathbb{E}_l\ \tilde{\to}\ \mathbb{E}_m$ and $e_m^n:\mathbb{E}_m\ \tilde{\to}\ \mathbb{E}_n$ is given by their composition $e_m^n\circ e_l^m:\mathbb{E}_l\ \tilde{\to}\ \mathbb{E}_n$.

Usually one investigates vector bundles with some additional structures. Further we will consider cases when the fibres of $\pi: \mathbb{E}\to M$ will be provided with these structures. For example such as: Hilbert space structure, $W^*$-algebra structure, lattice structure or $W^*$-algebra module structure.

Let $G$ be a groupoid with base set $B$. One defines the representation of $G$ on the vector bundle $\pi: \mathbb{E}\to M$ as a groupoid morphism:

\unitlength=5mm
\be\label{bundle}\begin{picture}(11,4.6)
    \put(1,4){\makebox(0,0){$G$}}
    \put(8,4){\makebox(0,0){$\G(\mathbb{E})$}}
    \put(1,-1){\makebox(0,0){$B$}}
    \put(8,-1){\makebox(0,0){$M$}}
    \put(1.2,3){\vector(0,-1){3}}
    \put(0.7,3){\vector(0,-1){3}}
    \put(8.2,3){\vector(0,-1){3}}
    \put(7.7,3){\vector(0,-1){3}}
    \put(2.5,4){\vector(1,0){4}}
    \put(2,-1){\vector(1,0){4.7}}
    \put(-0.2,1.4){\makebox(0,0){$\Ss$}}
    \put(2.2,1.4){\makebox(0,0){$ \Tt$}}
    \put(9.2,1.4){\makebox(0,0){$\Tt$}}
    \put(6.8,1.4){\makebox(0,0){$\Ss$}}
    \put(4.5,4.6){\makebox(0,0){$\phi$}}
    \put(4.5,-0.5){\makebox(0,0){$\varphi$}}
    \end{picture}\ee
\bigskip
\\
of $G$ into the structural groupoid $\G(\mathbb{E})$ of the bundle.

After these preliminary remarks let us consider the following two actions of $\G(\M)$ on the $W^*$-algebra $\M$:
 \ben[(i)]
\item the left action $L:\mathcal{G}(\M) *_l\M\to\M$ defined by
\be\label{left action} L_x y:=xy \ee for $(x,y)\in \G(\M)*_l\M:=\{(x,y)\in
\G(\M)\times \M;\quad r(x)=l(y)\}$;
\item  the right action $R:\mathcal{G}(\M) *_r\M\to \M$ defined by
\be\label{right action} R_x y:=yx \ee for $(x,y)\in \G(\M)*_r\M:=\{(x,y)\in
\G(\M)\times \M;\quad l(x)=r(y)\}$. \een
The moment map $\mu:\M\to\mathcal{L}(\M)$ (see Appendix C) for the left action $L$ (the right action $R$) is the left support map
$\mu:=l:\M\to \mathcal{L}(\M)$ (the right support map $\mu:=r:\M\to \mathcal{L}(\M)$)  defined in
(\ref{support}). Since the both actions are intertwined by the inverse map, i.e.
\be\label{circ} \iota\circ L_x=R_{\iota(x)}\circ
\iota\ee
we will restrict ourselves to the left action only. All statements concerning the right action
$R:\G(\M)*_r\M\to\M$ we obtain converting statements for the left action $L$   by (\ref{circ}).

Let us take $p,\ q\in \mathcal{L}(\M)$. According to the commonly accepted notation by $\G(\M)_p^q$ we denote the set $\Tt^{-1}(q)\cap\Ss^{-1}(p)$. For any $p\in \mathcal{L}(\M)$ one has the following inclusions:
\be\begin{array}{l} \Ss^{-1}(p)\subset r^{-1}(p)\subset \M p \\
\Tt^{-1}(p)\subset l^{-1}(p)\subset p\M\end{array}\ee
where $\M p$ ($p\M$) is the left (right) $W^*$-ideal generated by $p$. Note here that $\Ss=r|_{\G(\M)}$ and $\Tt=l|_{\G(\M)}$.

Now we consider the bundle $\pi:\mathcal{M}_R(\M)\to\mathcal{L}(\M)$ of right $\M$-modules over the lattice $\mathcal{L}(\M)$ with total space defined by
 \be\label{total} \mathcal{M}_R(\M):=\{(y,p)\in \M\times \mathcal{L}(\M):\quad p\ r(y)=r(y)\}\ee
 and bundle map $\pi:=pr_2$ as the projection on the second component of the product $\M\times \mathcal{L}(\M)$. The fibre $\pi^{-1}(p)$ over $p\in \mathcal{L}(\M)$ is the right ideal $p\M$ of $\M$ generated by the projection $p$. Any element $x\in \G(\M)_p^q$ defines by the left multiplication  an isomorphism $L_x:p\M\ \tilde{\to}\  q\M$ of the right $\M$-modules, i. e.
 \be\label{modul} L_x(ay)=L_x(a)y\ee
 for $a\in p\M$ and $y\in \M$. The $\M$-module isomorphisms $L:p\M\ \tilde{\to}\  q\M$, where $p,q\in \mathcal{L}(\M)$, form the groupoid $\G(\mathcal{M}_R(\M))$ of structural isomorphisms of the fibers of the bundle $\pi:\mathcal{M}_R(\M)\to\mathcal{L}(\M)$. One can show that $L=L_x$ for some $x\in \G(\M)_p^q$. Thus we have the following statement:
 \begin{prop} The structural groupoid $\G(\mathcal{M}_R(\M))$ of the bundle $\pi:\mathcal{M}_R(\M)\to\mathcal{L}(\M)$ is isomorphic to $\G(\M)$.
 \end{prop}
 Replacing $\mathcal{M}_R(\M)$ by  $\mathcal{M}_L(\M)$ and the  action $x\to L_x$  by the right action $x\to R_x$, where $x\in \G(\M)$, we obtain the anti-isomorphism of $\G(\M)$ with $\G(\mathcal{M}_L(\M))$. Using the above two representations  we obtain  a representation of  $\G(\M)$ on the bundle $\pi:\mathcal{A}(\M)\to \mathcal{L}(\M)$ of the $W^*$-subalgebras of  $\M$ with total space $\mathcal{A}(\M)$  defined by
 \be\label{A}
 \mathcal{A}(\M):=\{(y,p)\in \M\times \mathcal{L}(\M):\quad y\in p\M p\}
 \ee
 and the bundle map by $\pi:=pr_2$. The morphism $I:\G(\M)\to \G(\mathcal{A}(\M))$ of $\G(\M)$ into the structural groupoid $\G(\mathcal{A}(\M))$ of the bundle $\pi:\mathcal{A}(\M)\to\mathcal{L}(\M)$ is defined as follows
 \be\label{I} I_x:=R_{\iota(x)}\circ L_x: p\M p\to q\M q,\ee
where $x\in \G(\M)_p^q$.\\
 Note here that $\mathcal{J}(\M)\subset\mathcal{A}(\M)$ and the action   $I:\G(\M)*\mathcal{A}(\M)\to\mathcal{A}(\M)$ is an extension of the inner action $I:\G(\M)* \mathcal{J}(\M)\to\mathcal{J}(\M)$.  For $u\in \U(\M)_p^q$ we find that $I_u:p\M p\to q\M q$ is an isomorphism of $W^*$-subalgebras of $\M$. Thus we have

 \begin{prop}\label{2.6}
The inner action $I:\mathcal{U}(\M)* \mathcal{A}(\M)\to \mathcal{A}(\M)$ of the partial isometries
groupoid $\mathcal{U}(\M)$ on $\mathcal{A}(\M)$ preserves the  positivity, normality,
selfadjointness and the norm of the elements of the fibres of $\mathcal{A}(\M)$, i.e.: \ben[(i)]
\item $|I_ux|=I_u|x|$,
\item $xx^*=x^*x\ \ $ iff $\ \ (I_ux)^*(I_ux)=(I_ux)(I_ux)^*$
\item $x=x^*\ \ $ iff $\ \ (I_ux)^*=I_ux$
\item $\| I_ux\|=\|x\|$
\een for $(u,x)\in \mathcal{U}(\M)*\mathcal{A}(\M)$.
\end{prop}

Let $\M^+, \M^h$ and $\M^n$ denote the sets of positive, selfadjoint and
normal elements of $\M$ respectively. Let $S$ be the sphere in $\M$, i. e. $x\in
S$ if and only if $\|x\|=1$.  We conclude from Proposition \ref{2.6} that the
subsets $\J(\M)\cap\M^+$, $\ \J(\M)\cap\M^h$, $\ \J(\M)\cap\M^n$, $\
\J(\M)\cap S$ and $\ \J(\M)\cap \mathcal{U}(\M)$ are invariant with respect to
the inner action $I:\mathcal{U}(\M)* \J(\M)\to \J(\M)$. Let us also note that the
lattice of  projections $\mathcal{L}(\M)$ consists of the extreme points in
$\M^+\cap S$, e.g. see \cite{sakai}.

\begin{thm}
The actions $L:\mathcal{U}(\M) *_l \M\to \M$   and $R:\mathcal{U}(\M) *_r \M\to \M$
defined by (\ref{left action}) and (\ref{right action}) are free. Their orbits
are indexed by elements of the cone $\M^+$ of positive selfadjoint elements of
$\ \M$.
\end{thm}
\prf{ Since left and right actions are intertwined  by the inverse map
(\ref{inverse}) it is enough to consider the case of the left action $L$. Let us
assume that there are elements $u_1, \ u_2\in \mathcal{U}(\M)$ such that
$y:=u_1x=u_2x$ for $u_1^*u_1=vv^*=u_2^*u_2$, where $v\in \mathcal{L}(\M)$ is
defined by \be x=v|x|.\ee Since
$|y|^2=y^*y=x^*u_1^*u_1x=x^*v^*vx=x^*x=|x|^2$ we have $y=u_1v|y|=u_2v|y|$.
Thus from the uniqueness  of the polar decomposition(\ref{polar}) we obtain
$u_1v=u_2v$. The above gives
$u_1=u_1u_1^*u_1=u_1vv^*=u_2vv^*=u_2u_2^*u_2=u_2$. So, the left action
$L$ is free.

Taking the  polar decomposition  $x=v|x|$ of $x\in \M$  we obtain that
$v^*x=v^*v|x|=|x|\in \M^+$. So any orbit $\mathcal{O}_x$ of $\mathcal{U}(\M)$ intersects
$\M^+$. If $x,y \in \mathcal{O}_x\cap \M^+$ then $x=|x|$, $y=|y|$ and $|y|=u|x|$ for some
$u\in \mathcal{U}(\M)$. Thus  from uniqueness  of the polar decomposition
we obtain $x=|x|=|y|=y$.}

\bigskip

For the sake of completeness and the further applications let us consider the
actions
\be\label{actions*}\begin{array}{l} L_{*}:\ \mathcal{U}(\M) *_{l_*} \M_*\to\M_*\\
 R_{*}:\ \mathcal{U}(\M) *_{r_*} \M_*\to\M_*\end{array}\ee
of $\mathcal{U}(\M)$ on $\M_*$ which are  predual to the actions $L:
\mathcal{U}(\M)\times\M\to\M$ and $\ R:\mathcal{U}(\M)\times\M\to\M$
respectively.

Refering to \cite{takesaki} we recall,  that the left predual action $L_{*}:\
\M\times\M_*\to\M_*$ (respectively the right predual action $R_{*}:\
\M\times\M_*\to\M_*$) of $W^*$-algebra $\M$ on the predual Banach space
$\M_*$ is defined by: \be \label{action} \langle x,L_{*a}\omega\rangle:=\langle
xa,\omega\rangle\ \ \ \ ({\rm respectively}\ \langle x,R_{*a}\omega
\rangle:=\langle ax,\omega\rangle)\ee for any  $x\in\M$, where $a\in \M$ and
$\omega\in \M_*.$ So, one has \be\label{sprz}(L_{*a})^*=R_a, \ \ \ {\rm and}\ \ \
(R_{*a})^*=L_a.\ee For any element $\omega\in \M_*$ one takes the closed
left  invariant subspace $[\M\omega]\subset \M_*$  (respectively the right
invariant subspace $ [\omega\M]\subset \M_*$) generated from $\omega$ by
the left (respectively right) action of $\M$ (\ref{action}). The annihilator
$[\M\omega]^0\subset \M$ of the Banach subspace $[\M\omega]\subset \M$
 is the right  $W^*$-ideal in $\M$. Similarly the
annihilator $[\omega\M]^0\subset \M$ of the Banach subspace
$[\omega\M]\subset \M$
 is the left  $W^*$-ideal in $\M$.
Thus there exist the orthogonal projections $e, f \in \M$  such that
$[\M\omega]^0=e\M$ and $ [\omega\M]^0=\M f$. The projection $e$  is the
greatest one of all the projections $q\in \M$ such that $R_{*q}\omega =0$.
Similarly the projection $f$ is the greatest one of all the projections $q\in \M$
such that $L_{*q}\omega =0$. Thus one defines the map \be r_*(\omega):=1-e
\ee and \be \ l_*(\omega):=1-f,\ee
 where $(1-e)$ and  $  (1-f)$  are  the least projection with the property
 $R_{*(1-e)}\omega=\omega$ and $  L_{*(1-f)}\omega=\omega$ respectively.
 The projections $r_*(\omega)$ and  $ l_*(\omega)$  are  called, respectively, the right
 support projection  and
 the left support projection of $\omega\in\M_*$. It follows from the polar decomposition
(e.g. see \cite{takesaki}) \be \omega=L_{*v}|{\omega}|\ee of $\omega\in \M_*$,
where $v\in \mathcal{U}(\M)$ and $|\omega|\in \M^+_*$, that
$$r_*(\omega)=v^*v \ \ {\rm and}\ \ l_*(\omega)=vv^*.$$

Considering $r_*:\M_*\to \mathcal{L}(\M)$\ and $ l_*:\M_*\to \mathcal{L}(\M)$ as  moment maps
we define  the actions (\ref{actions*})  by
\be\label{rag} \mathcal{U}(\M)*_{r_*}\M_*\ni(u,\omega)\mapsto
R_{*u}\omega\in\M_*\ee and  \be \label{lag}
\mathcal{U}(\M)*_{l_*}\M_*\ni(u,\omega)\mapsto L_{*u} \omega\in\M_*,\ee
respectively, where
$$\mathcal{U}(\M) *_{r_*}\M_*=\{(u,\omega)\in \mathcal{U}(\M) \times\M_*;\
\ \ \Tt(u)=uu^*=r_*(\omega)\}$$ and
$$ \mathcal{U}(\M) *_{l_*}\M_*=\{(u,\omega)\in \mathcal{U}(\M) \times\M_*;\ \ \
\Ss(u)=u^*u=l_*(\omega)\}.$$

Let us define the  bundle $\pi_*:\mathcal{A}_*(\M)\to \mathcal{L}(\M)$ predual to the bundle of the $W^*$-algebras $\pi:\mathcal{A}(\M)\to\mathcal{L}(\M)$. In this case the total space is the following
\be\label{A*} \mathcal{A}_*(\M):=\{(\omega,p)\in\M_*\times \mathcal{L}(\M):\ r_*(\omega)=p\ r_*(\omega)\ {\rm and}\ l_*(\omega)=l_*(\omega)p\}\ee
and the bundle map $\pi_*$ is the projection of $(\omega,p)\in \M_*\times \mathcal{L}(\M)$ on the second component. Note that one can identify the fibre $\pi^{-1}_*(p)=(R_{*p}\circ L_{*p})(\M)$, $\ p\in\mathcal{L}(\M)$, with the Banach space $(p\M p)_*$ predual to subalgebra $p\M p$.

We define the predual inner action
\be\label{inner*} I_*:\G(\M)*\mathcal{A}_*(\M)\to \mathcal{A}_*(\M)\ee of the groupoid $\G(\M)$  on $\mathcal{A}_*(\M)$
where $ \G(\M)*\mathcal{A}_*(\M):=\{(x,(\omega,p)):\ \Ss(x)=p\}$
and the bundle map $\pi_*:\mathcal{A}_*(\M) \rightarrow \mathcal{L}(\M)$ is the
moment map.

Now we  define the following subbundles of $\pi_*:\mathcal{A}_*(\M)\to \mathcal{L}(\M)$. The subbundle $\pi_*:\mathcal{J}_*(\M)\to\mathcal{L}(\M)$  whose total space is defined by
\be\label{bundle1}\mathcal{J}_*(\M):=\{(\omega,p)\in\mathcal{A}_*(\M):\quad l_*(\omega)=r_*(\omega)=p\}.\ee
The subbundle of selfadjoint normal functionals $\pi_*:\mathcal{A}^h_*(\M)\to\mathcal{L}(\M)$ for which the total space is
\be\label{bundle3}\mathcal{A}^h_*(\M):=\{(\omega,p)\in\mathcal{A}_*(\M): \ \omega^*=\omega\}.\ee
The subbundle of positive normal functionals $\pi_*:\mathcal{A}^+_*(\M)\to\mathcal{L}(\M)$ with  the set
\be\label{bundle2}\mathcal{A}^+_*(\M):=\{(\omega,p)\in\mathcal{A}_*(\M):\ \omega^*=\omega>0\}.\ee
as the total space.\\

\bigskip

When we restrict the action (\ref{inner*}) to subgroupoid $\mathcal{U}(\M)\subset\G(\M)$ we obtain the following statement.

\begin{prop}\label{pred}
For the predual inner action $I_*:\mathcal{U}(\M)* \mathcal{A}_*(\M)\to \mathcal{A}_*(\M)$ of the
partial isometries groupoid $\mathcal{U}(\M)$ on $\mathcal{A}_*(\M)$ one has:
\be\label{Aaa}  \omega=\omega^*\ \ {\rm iff }\ \
(I_{*u}\omega)^*=I_{*u}\omega\ee \be\label{Bb}  \| I_{*u}\omega\|=\|\omega\|\ee
\be\label{Cc}|I_{*u}\omega|=I_{*u}|\omega|\ee
for $(u,(\omega,p))\in \mathcal{U}(\M)*\mathcal{A}_*(\M)$, i.e. the subbundles (\ref{bundle1}), (\ref{bundle2}), (\ref{bundle3}) are  invariant with respect to this action.

\end{prop}
\prf {In order to prove (\ref{Aaa})  we note that for $\omega\in\pi_*^{-1}(u^*u)$ we have
$\langle\omega^*,x\rangle:=\overline{\langle\omega,x^*\rangle}$, where
$x\in\M$, so we obtain
$$(I_{*u}\omega)^*=I_{*u}\omega^*.$$
Thus and from $I_{*u^*u}\omega=\omega$
 we have that  $\omega=\omega^*$ iff $(I_{*u}\omega)^*=I_{*u}\omega.$
\newline Since $\norm u =1$ and $L_{*u^*u}\omega=\omega$ one has \be
\norm{L_{*u}\omega}\leqslant\norm\omega\leqslant\norm{L_{*u}\omega}.\ee
Similarly we prove that $\norm{R_{*u}\omega}=\norm\omega$. Thus we have
(\ref{Bb}).

 Let us assume that $\M_*\ni \omega \geqslant 0$ then for any  $x\in
\M^+$ one has $u^*xu\in \M^+$ and
$$\langle I_{*u}\omega, x\rangle=\langle\omega, u^*xu\rangle\geqslant 0.$$ Thus we find that
$I_{*u}\omega\geqslant 0$ iff $\omega\geqslant 0$. If $\omega=\omega^*$
then for any $x\in \M$ we have
$$\langle( I_{*u}\omega)^*, x\rangle = \overline{\langle\omega, u^*xu\rangle}=\langle\omega^*,
(u^*xu)^*\rangle =\langle I_{*u}\omega^*, x\rangle.$$ The above shows that inner action commutes with conjugation
 operation.

 Let us take the polar
decomposition of $\omega\in\pi_*^{-1}(u^*u)$ \be \omega= L_{*v}|\omega|.\ee We
note that the polar decomposition of $I_{*u}\omega\in\pi_*^{-1}(u^*u)$,  where
$v^*v\leqslant u^*u$ and $vv^*\leqslant u^*u$, is given by \be\label{AA} I_{*u} \omega=L_{uvu^*}|I_{*u}
\omega|. \ee From  (\ref{AA})  it follows that
$$|I_{*u} \omega|=L_{(uvu^*)^*}I_{*u} \omega.$$ Thus for any $x\in \M$ we have
 $$\langle| I_{*u}\omega|, x\rangle=\langle L_{(uvu^*)^*} I_{*u}\omega,
x\rangle=\langle I_{*u}\omega, xuv^*u^*\rangle=$$ $$ =\langle \omega,
u^*xuv^*u^*u\rangle=\langle \omega, u^*xuv^*\rangle=\langle L_{*v^*}\omega,
u^*xu\rangle=\langle |\omega|, u^*xu\rangle=$$ $$ =\langle I_{*u}|\omega|,
x\rangle.$$ Thus we obtain (\ref{Cc}) }

Now basing on GNS construction we define the pre-Hilbert bundle $\pi:\mathbb{E}\to \M_*^+$ over the cone of the positive normal states $\M _*^+$. The total space $\mathbb{E}$ and bundle projection we define as follows
\be\label{preHil} \mathbb{E}:=\{(x,\omega)\in \M\times \M_*^+:\quad xs_*(\omega)=x\}\ee
and $\pi:=pr_2|_{\mathbb{E}}$.\\
Since for $\omega\in \M _*^+$ one has $\mathbb{E}_{\omega}=\pi^{-1}(\omega)=\M s_*(\omega)$ the scalar product
 \be\mathbb{E}_{\omega}\times \mathbb{E}_{\omega}\ \ni\ (x,y)\mapsto \langle
x|y\rangle_{\omega}:=\langle\omega,x^*y\rangle\ \in\ \mathbb{C}\ee
is non degenerate. Thus it defines the pre-Hilbert space structure on $\mathbb{E}_{\omega}$. Note here that $\langle\omega,x^*x\rangle=0$ if and only if $x\in \M(1-s_*(\omega))$.

Completing $\mathbb{E}_{\omega}$ with respect to the norm  $\norm{x}_{\omega}:=\langle\omega,x^*x\rangle^{\frac{1}{2}}$ we obtain the bundle  $\ol{\pi}:\ol{\mathbb{E}}\to \M_*^+$ of Hilbert spaces. Since for the clear reasons we will call this bundle the GNS bundle.
\begin{thm}\label{rep}
One has a   faithful representation \unitlength=5mm

\be\label{diag}\begin{picture}(11,4.6)
    \put(0.8,4){\makebox(0,0){$\mathcal{U}(\M)*_{}\M_*^+$}}
    \put(8,4){\makebox(0,0){$\G(\ol{\mathbb{E}})$}}
    \put(1,-1){\makebox(0,0){$\M_*^+$}}
    \put(8,-1){\makebox(0,0){$\M_*^+$}}
    \put(1.2,3){\vector(0,-1){3}}
    \put(0.7,3){\vector(0,-1){3}}
    \put(8.2,3){\vector(0,-1){3}}
    \put(7.7,3){\vector(0,-1){3}}
    \put(3,4){\vector(1,0){3}}
    \put(2,-1){\vector(1,0){4.7}}
%\put(8.5,2){\vector(1,0){3.5}}
%\put(0.5,2){\vector(-1,0){3.5}}
    \put(-0.2,1.4){\makebox(0,0){$\Ss$}}
    \put(2.2,1.4){\makebox(0,0){$\Tt$}}
    \put(9.2,1.4){\makebox(0,0){$\Tt$}}
    \put(6.8,1.4){\makebox(0,0){$\Ss$}}
    \put(4.5,4.4){\makebox(0,0){$\phi$}}
    \put(4.5,-0.5){\makebox(0,0){$id$}}
%\put(-1.2,2.5){\makebox(0,0){$\tilde\varepsilon$}}
%\put(10.5,2.5){\makebox(0,0){$\varepsilon$}}
    %\put(-4.5,2){\makebox(0,0){$G*M$}}
    %\put(12.7,2){\makebox(0,0){$G$}}
    \end{picture}\ee
\bigskip
\\
of the right action of groupoid $\mathcal{U}(\M)*_{}\M_*^+$ on the GNS Hilbert spaces
bundle $\ol{\pi}:\ol{\mathbb{E}}\to \M_*^+$  with the fibres isomorphisms $\phi(u,\omega): \ol{\mathbb{E}}_{\omega}\to \ol{\mathbb{E}}_{I_{*u}\omega}$ defined as follows
 \be\label{phi}\phi(u,\omega)(xs_*(\omega),\omega):=(xs_*(\omega)u^*, I_{*u}\omega)\ee
\end{thm}
\begin{proof}
The following sequence of equalities \be\label{D}\begin{array}{l}
\langle \phi(u,\omega)xs_*(\omega)|\phi(u,\omega)ys_*(\omega)\rangle_{I_{*u}\omega }=\\
=\langle{I_{*u}\omega,(xs_*(\omega)u^*)^*ys_*(\omega)u^*}\rangle=\\
=\langle{\omega,u^*us_*(\omega)x^*ys_*(\omega)u^*u}\rangle=\ \ \ \ \ \ \ \ \ \ \ \ \ \ \ \ \ \ \ \ \ \ \\
\ \ \ \ \ \ \ \ \ \ \ \ \
=\langle{\omega,(xs_*(\omega))^*ys_*(\omega)}\rangle=\langle{xs_*(\omega)|ys_*(\omega)}\rangle_{\omega}\end{array}\ee
shows that
$\phi(u,\omega): {\mathbb{E}}_{\omega}\to {\mathbb{E}}_{I_{*u}\omega}$ extends to isomorphism of Hilbert spaces.

For elements $(u,\omega),\ (v,\lambda)\in \mathcal{U}(\M)*_{}\M_*^+$ such
that $\Tt(v)=\Ss(u)$, i.e. $\ \omega=I_{*v}\lambda$ we have
$$\phi((u,\omega)(v,\lambda))(xs_*(\lambda,)\lambda)=\phi(uv,\lambda)(xs_*(\lambda),\lambda)=
(xs_*(\lambda)(uv)^*,I_{*uv}\lambda)=$$
$$=\phi(u,\omega)(xs_*(\lambda)v^*,I_{*u}\lambda )= (\phi(u,\omega)\circ(\phi(v,\lambda))(xs_*(\lambda),\lambda))$$
for any $(xs_*(\lambda),\lambda)\in {\mathbb{E}}_\lambda$. Thus we obtain
\be\label{B}\phi((u,\omega)(v,\lambda))=\phi(u,\omega)\circ\phi(v,\lambda).\ee
We recall that  $(u,\omega)(v,\lambda)$ is  the product of $(\mathcal{U}(\M)*_{}\M_*^+)^{(2)}$
defined by (\ref{*prod}).

 One can easily check that
 %$$\iota_{_\Gamma}(\omega,\varphi_{_\Gamma}(u,\omega),u\omega u^*)=(u\omega u^*,\varphi_{_\Gamma}(u^*,u\omega u^*),\omega)$$
 %for any element $(\omega,\varphi_{_\Gamma}(u,\omega),u\omega u^*)\in \mathcal{U}(\Gamma)$.\\
%It means that
for $\phi(u,\omega)$ and $\phi(u^*,I_{*u}\omega)$ we have
$$(\phi(u^*,I_{*u}\omega)\circ\phi(u,\omega))( xs_*(\omega),\omega)=\phi(u^*,I_{*u}\omega)(xs_*(\omega)u^*,I_{*u}\omega)=$$
$$=(xs_*(\omega)u^*u, I_{*u^*u}\omega)=(xs_*(\omega), \omega)$$
for any $(xs_*(\omega),\omega)\in \mathbb{E}_\omega$. The above
shows that $$\phi(u^*,I_{*u}\omega)\circ\phi(u,\omega)=id|_{\mathbb{E}_\omega}.$$ In the similar
way we prove that
$$\phi(u,\omega)\circ\phi(u^*,I_{*u}\omega)=id|_{\mathbb{E}_{I_{*u}}\omega}.$$ Thus we get
\be\label{C}(\phi(u,\omega))^{-1}=\phi(u^*,I_{*u}\omega).\ee

 For any
$(u,\omega)\in \mathcal{U}(\M)*_{}\M_*^+$ one has
$$(id\circ \Ss)(u,\omega)=id(\omega)=\omega,$$
$$(id\circ \Tt)(u,\omega)=id(I_{*u}\omega)=I_{*u}\omega$$
and
$$(\Ss\circ \phi)(u,\omega)(xs_*(\omega),\omega)=\Ss(\phi(u,\omega)(xs_*(\omega),\omega))=\omega, $$
$$(\Tt\circ \phi)(u,\omega)(xs_*(\omega),\omega)=\Tt(\phi(u,\omega)(xs_*(\omega),\omega))={I_{*u}}\omega,$$
which shows that $id\circ \Ss=\Ss\circ \phi$ and $id\circ \Tt=\Tt\circ
\phi$, i. e. the diagram (\ref{diag}) is commutative. The above
shows that $\phi$ is a groupoid morphism.

From \be\label{wierna}
\phi(u,\omega)=\phi(u',\omega').\ee
 we find that\\
$$\omega=\omega'$$
\be\label{wierna1}s_*(\omega)=u^*u=u'^*u'=s_*(\omega')\ee and
\be\label{wierna2}xs_*(\omega)u^*=xs_*(\omega)u'^*\ee for any $x\in\M$.
Setting $x=s_*(\omega)$ in (\ref{wierna2})  we prove  that from (\ref{wierna}) it
follows $$(u,\omega)=(u',\omega ').$$ Thus $\phi$ is the faithful
morphism of groupoids.
\end{proof}

In order to obtain a faithful $W^*$-representation of $\M$, see \cite{sakai}, we recall that $\mathbb{E}_\omega$ is a left $W^*$-ideal of $\M$.  Hence one has $W^*$-representation $\ol{\rho_\omega}:\M\to L^\infty(\ol{\mathbb{E}_\omega})$ of $\M$ in the $W^*$-algebra of bounded operators on Hilbert space  $\ol{\mathbb{E}_\omega}$, defined by the continuous extension of
\be \rho_\omega(x)ys_*(\omega):=xys_*(\omega),\ee
 where $x\in\M$ and $ys_*(\omega)\in \mathbb{E}_\omega$.

 Let us denote by $L^2\Gamma(\mathbb{E},\M^+_*)$ the Hilbert space of square summable sections $\psi:\M_*^+\to\ol{\mathbb{E}}$
 \be\label{258}\sum_{\omega\in\M^+_*}\norm{\psi(\omega)}^2<\infty\ee
  of the GNS bundle.

 The direct sum
 \be\label{ro}\ol{\rho}:=\bigoplus_{\omega\in\M^+_*}\ol{\rho_\omega}\ee
 is a faithful $W^*$-representation $\ol{\rho}:\M\to L^\infty(L^2\Gamma(\mathbb{E},\M^+_*))$ of $\M$ in the Hilbert space $L^2\Gamma(\mathbb{E},\M^+_*)$. Recall that the $^*$-homomorphism of $W^*$-algebras is a $W^*$-homomorphism  if it is a map $\phi:\M_1\to\M_2$ continuous with respect to the $\sigma(\M_1,\M_{1*})$-topology and  $\sigma(\M_2,\M_{2*})$-topology.

 Let $\ol\rho(\M)'$ be the commutant of $\ol\rho(\M)$ in the operator algebra $L^\infty(L^2\Gamma(\mathbb{E},\M^+_*))$.
 \begin{thm}\label{propLambda}
 There exists a  groupoid monomorphism
 \be\label{Lambda}\Lambda:\mathcal{U}(\M)*_{}\M_*^+\to \U(\ol\rho(\M)')\ee
 of the action groupoid $\mathcal{U}(\M)*_{}\M_*^+$ into the groupoid of partial isometries $\U(\ol\rho(\M)')$ of the $W^*$-algebra $\ol\rho(\M)'$.\end{thm}
 \prf{Any element $e_{\omega_1}^{\omega_2}\in \G(\mathbb{E})_{\omega_1}^{\omega_2}\subset \G(\M)$ of the structural  groupoid  which maps $e_{\omega_1}^{\omega_2}:\ol{\mathbb{E}}_{\omega_1}\to \ol{\mathbb{E}}_{\omega_2}$ has the  extension  to a partial isometry of the Hilbert space $L^2\Gamma(\mathbb{E},\M^+_*)$. Since Hilbert subspaces $\ol{\mathbb{E}}_{\omega_1}$ and $\ol{\mathbb{E}}_{\omega_2}$ of  $L^2\Gamma(\mathbb{E},\M^+_*)$ are invariant with respect to the representation $\ol\rho$ we conclude that $e_{\omega_1}^{\omega_2}\in\U(\ol\rho(\M)') $. Hence we obtain a monomorphism $\iota: \G(\mathbb{E})\hookrightarrow \U(\ol\rho(\M)')$. Thus  Proposition \ref{rep} implies that $\Lambda=\iota\circ \phi$ is a groupoid monomorphism.
 }

Finally let us note that the projection \be\label{coscos} pr_1:\mathcal{U}(\M)*_{}\M_*^+\to \U(\M)\ \cong\ \U(\ol\rho(\M)')\ee
 of $\mathcal{U}(\M)*_{}\M_*^+$ on the first component of the product $\mathcal{U}(\M)\times\M_*^+$ defines a morphism of the groupoids. The map (\ref{coscos}) is an epimorphism of groupoid if and only if $\M$ is a $\sigma$-finite $W^*$-algebra.

\section{Inverse semigroups in $\mathcal{U}(\M)$}

Usually one  finds in specialized papers the constructions of the operator algebras from the representation of given inverse semigroup. See for example \cite{pater}, where also one can find other references. Here we will follow conversely obtaining inverse semigroup from given $W^*$-algebra.

Let us recall that the \textbf{inverse semigroup} is a semigroup S in which for each element $s\in S$ exists a unique element $t\in S$ such that
\be\label{invers} sts=s \qquad tst=t.\ee The element $t$ is denoted by $s^*$ and the map $s\to s^*$ is involution on $S$.
An important role  plays the subset $E(S)\subset S$ of idempotents of the inverse semigroup $S$. Its easy to check that for any $s\in S$  elements $ss^*$ and $s^*s$ belong to $E(S)$. For any elements $e, f \in E(S)$ we have $e=e^*$ and $ef=fe$.

 It is a well-known fact in the theory of the inverse semigroups, which follows directly from the
Vagner-Preston  theorem \cite{howie}, that any inverse semigroup
$\mathcal{S}$ is composed  by the partial isometries in the Hilbert space
$l^2(\mathcal{S})$, see \cite{pater}. On the other hand any $W^*$-algebra
$\M$ is also realized by bounded operators in some Hilbert space $\H$, i.e.
$\M\subset L^\infty (\H)$. These facts motivate us to investigate the inverse
semigroups which are composed by partial isometries of $W^*$-algebra, i.e. $S\subset \U(\M)$.

\begin{rem}\label{rem1}
The product $uv$ of two partial isometries $u,v\in\mathcal{\U}(\M)$ is a partial
isometry if and only if the initial projection of $u$ and the final projection of $v$
 commute, see \cite{H-W}. \end{rem}

 Let us assume that partial isometries $u, v\in S$ satisfy property (\ref{invers}). Then one has
\be\label{invers2}u^*=u^*uu^*=u^*uvuu^*=u^*uvv^*vv^*vuu^*=\qquad\qquad\ee $$\qquad\qquad=vv^*u^*uvuu^*v^*v
=v(v^*(u^*uu^*)v^*)v=v.$$
The equality (\ref{invers2}) shows that for $S\subset \U(\M)$ the element $v$ inverse to $u$ in sense of (\ref{invers}) is unique and it is defined as $u^*$.

So, let us assume that $\mathcal{S}\subset \M$. Hence the set $E$ of
idempotents of $\mathcal{S}$ is given by $E=\Ss(\mathcal{S})=\Tt(\mathcal{S})\
\subset\ \mathcal{L}(\M)$ and $\mathcal{S} \subset \mathcal{U}(E)$ where \be\label{UE}
\mathcal{U}(E):=\Ss^{-1}(E)\cap\Tt^{-1}(E) \ee is the full subgroupoid of the
groupoid of partial isometries $\mathcal{U}(\M)$. Since the product $uv$ in
$\mathcal{S}$  is defined for all $(u,v)\in \mathcal{S}\times \mathcal{S}$ it is
defined also for $(u,v)\in\mathcal{S}^{(2)}=\{(u,v):\ \ \Ss(u)=\Tt(v)\}$. Therefore,
the inverse semigroup $\mathcal{S}\subset \mathcal{U}(E)\subset
\mathcal{U}(\M)$ can be considered as a subgroupoid of the groupoid
$\mathcal{U}(\M)$. In this context  the question when $\mathcal{U}(E)$ is the
inverse semigroup arises.

\begin{prop}\label{propUE}
The subgroupoid $\mathcal{U}(E)$  complemented by $\{0\}$ is an inverse semigroup
if and only if: \ben[(i)]
\item for  $p,q\in E$ one has $pq=qp\in E$,
\item for  $u\in \mathcal{U}(E)$ and $p\in E$ one has $upu^*\in E$ or $upu^*=0$.
\een
\end{prop}
\prf{Let us assume that $\mathcal{U}(E)$ is an inverse semigroup. Then
$E\subset \mathcal{U}(E)$ consists of idempotents of $\mathcal{U}(E)$ and is
a commutative subsemigroup of $\mathcal{U}(E)$. If $p\in E$ and $u\in
\mathcal{U}(E)$ then $up\in \mathcal{U}(E)$. Thus $upu^*=up(up)^*\in E$.

Now, let us assume that properties (i) and (ii) of the proposition are valid. Then
for $u,v\in \mathcal{U}(E)$ we have
$$uv(uv)^*uv=uvv^*u^*uv=uu^*uvv^*v=uv.$$ So, $uv$ is a partial isometry and
$uv(uv)^*=uvv^*u^*\in E$. This shows that $uv\in \mathcal{U}(E)$. }

 From the
Proposition \ref{propUE} we conclude:
\begin{cor}
The partial isometries groupoid $\mathcal{U}(\M)$ is an inverse semigroup if
and only if $W^*$-algebra $\M$ is abelian.
\end{cor}
\prf{If $\mathcal{U}(\M)$ is an inverse semigroup then the lattice $\mathcal{L}(\M)$ is a
Boolean one and $pq=qp$ for $p,q\in \mathcal{L}(\M)$. From this and from the spectral
theorem we obtain that $\M$ is an abelian $W^*$-algebra. If $\M$ is abelian
and $u,v\in \mathcal{U}(\M)$ then $$uv(uv)^*uv=uvv^*u^*uv=uu^*uvv^*v=uv.$$
It means that $\mathcal{U}(\M)$ is closed with respect to multiplication, i.e.
$uv\in \mathcal{U}(\M)$. So, $\mathcal{U}(\M)$ is an inverse semigroup.}

\begin{example}
Let us consider a family $\{p_i\}_{i\in I}$ of mutually orthogonal projections
$p_ip_j=\delta_{ij}p_i$ in $\M$. We define
$\mathcal{U}_{ij}:=\Tt^{-1}(p_i)\cap\Ss^{-1}(p_j)$, i.e. $u_{ij}\in \mathcal{U}_{ij}$
iff $u_{ij}^*u_{ij}=p_j$ and $u_{ij}u_{ij}^*=p_i$. So, for $u_{ij}\in
\mathcal{U}_{ij}$ and $u_{kl}\in \mathcal{U}_{kl}$ one has
\be\label{ex1}u_{ij}u_{kl}=\delta_{jk}u_{ik}u_{kl}\ee and  \be\label{ex1a}
u^*_{ij}\in \mathcal{U}_{ji}.\ee In the case when $\mathcal{U}_{ij}=\emptyset$
we will assume in (\ref{ex1}) and (\ref{ex1a}) by definition that $u_{ij}=0$.
Concluding we see that $\mathcal{U}(E)\cup \{0\}$, where $E=\{p_i\}_{i\in I}$, is
an inverse semigroup.

\end{example}

As a specialization of the above abstract example of the inverse semigroup let us take $I=\M_*^+$ and define $\U_{\omega '\omega}:=\U(\M)_{\omega}^{\omega '}$ for $\omega, \omega '\in\M_*^+$. Here we consider $\U(\M)_{\omega}^{\omega '}$ as a set of partial isometries in Hilbert space $L^2\Gamma(\mathbb{E},\M_*^+)$, see definition (\ref{258}).

\begin{example}
Let $V_{ij}\subset \mathcal{U}_{ij}$ be such that  $\bigcup_{i,j\in
I}V_{ij}\cup\{0\}$ is closed with respect to  (\ref{ex1}) and (\ref{ex1a}).
 Then $\bigcup_{i,j\in I}V_{ij}\cup\{0\}$ is the
 inverse subsemigroup of $\mathcal{U}(\{p_i\}_{i\in I})\cup\{0\}$.
\end{example}

\begin{example}
 In the case when $V_{ij}$ are one element subsets we will denote  the inverse semigroup
 $\bigcup_{i,j\in I}V_{ij}$
 by $\mathcal{V}(\{p_i\}_{i\in I})$. Note also that $\mathcal{V}(\{p_i\}_{i\in I})$ is the
 $(I\times I)$-matrix unit in $p\M p$, where $p=\sum_{i\in I} p_i$. For the definition of the
 $(I\times I)$ - matrix unit see e.g.
\cite{kadison}, \cite{takesaki}. In the below  we will call $\mathcal{V}(\{p_i\}_{i\in
I})$ the \emph{matrix unit inverse semigroup}.

\end{example}

The notion of the  $(I\times I)$  - matrix unit plays an important role in studying
of $W^*$-algebra structure for particular of I type $W^*$-algebras, e.g. see
\cite{kadison}. Recall that for any central projection $z$ of a $W^*$-algebra $\M$ of type I
 one has decomposition $z=\Sigma_{i\in I}p_i$ given by a
family $\{p_i\}_{i\in I}$ of mutually orthogonal, equivalent abelian projections.
The cardinal number $\alpha:=\sharp I$ does not depend on a choice of
$\{p_i\}_{i\in I}$. Hence we conclude that for any $\alpha$-homogeneous direct
summand $z\M$ of the I type $W^*$-algebra $\M$ one has abelian $(I\times
I)$-matrix unit inverse semigroup  $\mathcal{V}(\{p_i\}_{i\in I})$. The inverse
semigroups  $\mathcal{V}(\{p_i\}_{i\in I})$ are isomorphic for different choices
of the decompositions $z=\Sigma_{i\in I}p_i$.

\bigskip

Let $B_{part}(I)$ be the inverse semigroup of partial bijections of the set of
indices $I$ which appeared in the above examples. Let $\mathfrak{I}(I)\subset
B_{part}(I)$ be the subset of partial bijections $\varphi:A \ \tilde{\to}\  B$, where
$A, B\subset I$, such that \be\label{parbij} U_{\varphi{(i)}i}\neq\emptyset.\ee
 for each $ i\in A$.\\
We denote by $\mathcal{U}(I)\subset \mathcal{U}(\M)$  the subset of partial
isometries defined by \be \label{par isom} u_\varphi :=\sum_{i\in A}
u_{\varphi{(i)}i},\ee where $u_{\varphi{(i)}i}\in U_{\varphi(i)i}$.

\begin{thm}\label{3.3}\ben[(i)]
\item
The subset $\mathcal{U}(I)\subset \mathcal{U}(\M)$ is an inverse semigroup of
partial isometries.
\item
The subsemigroup $E$ of idempotents of $\mathcal{U}(I)$ is a semilattice
 of orthogonal projections which are defined by
\be p_A:=\sum_{i\in A}p_i,\qquad A\subset I. \ee
\item
The map \be\label{surj} \phi: \mathcal{U}(I)\ \ni\ u_\varphi\to \varphi \in
\mathfrak{I}(I) \ee is a surjective morphism of the  inverse semigroups. \een
\end{thm}
\prf{  \ben[(i)]
\item  Taking another partial bijection $\psi:C\ \tilde{\to}\  D$ of
$C,D\subset I$  and the corresponding partial isometry $u_\psi$  we find that
\be\label{mno} u_\psi u_\varphi =\left(\sum_{j\in C}
u_{\psi{(j)}j}\right)\left(\sum_{i\in A} u_{\varphi{(i)}i}\right) =\ \ \ \ \ \ \ \ \ \ \ \ \ \ \ \ \
\ \ \ee $$ \ \ \ \ \ \ \ \ \ \ \ \ \ \ \ \ \ \ \ \ \ \ \ = \sum_{k\in \varphi^{-1}(B\cap C)}
u_{(\psi\circ\varphi)(k)\varphi(k)} u_{\varphi(k)k}, $$ \be\label{par inv}
u^*_\varphi= \sum_{i\in A}u^*_{\varphi(i)i}=\sum_{j\in B}u^*_{j\varphi^{-1}(j)},\ee
where \be \label{bij} \psi\circ\varphi:\varphi^{-1}(B\cap C)\ \tilde{\to}\  \psi(B\cap
C)\ee and \be \varphi^{-1}:B \ \tilde{\to}\  A.\ee  The product
$u_{(\psi\circ\varphi)(k)\varphi(k)} u_{\varphi(k)k}$ belongs to
$U_{(\psi\circ\varphi)(k) k}$ and $u^*_{\varphi(i)i}\in U_{\varphi^{-1}(j)j}$, where
$j=\varphi(i)$. Thus $u_\psi u_\varphi\in \mathcal{U}(I)$  and $u^*_\varphi\in
\mathcal{U}(I)$.

\item It follows from (\ref{mno}) and (\ref{par inv}) that
$$u^*_\varphi u_\varphi =u_{\varphi^{-1}}u_\varphi=u_{id_A}=\sum_{i\in A}p_i$$
where $id_A:A\to A$ is identity partial bijection. Similarly we can show that
$u_\varphi u^*_\varphi=\sum_{i\in B}p_i$.

\item From (\ref{mno}) and (\ref{par inv}) follows that $\mathfrak{I}(I)$ is closed with respect
to the superposition of partial bijections. The fact, that the map  $\phi$
defined in (\ref{surj}) is a surjective morphism, follows from  (\ref{mno}) and
(\ref{par inv}) and definition of $\mathfrak{I}(I)$.

\een}

\begin{rem}Taking partial isometries $u_\varphi\in\mathcal{U}(I)$ where $\varphi$ is a bijection between one elements sets  we find that  the matrix unit inverse semigroup  $\mathcal{U}(\{p_i\}_{i\in I})$ is an  inverse subsemigroup   of $\mathcal{U}(I)$.
\end{rem}

 In order to obtain an  interesting examples of the inverse semigroup $\mathcal{V}(\{p_i\}_{i\in I})$,
where $I=\{1,2,...,N\}$ or $I=\mathbb{N}$,
 let us take a properly infinite $W^*$-algebra $\M$ . It follows from the Halving Theorem
 (see \cite{kadison}) that there exists a family $\{p_i\}_{i\in I}$ of mutually orthogonal
  projections in $\M$ such that
  \be\label{cuntz} p_i\sim \mathbf{1}\quad {\rm  and}\quad \sum_{i\in I}p_i=\mathbf{1},\ee
 where $I=\{1,2,...,N\}$  or $I=\N\cup\{\infty\}$. Since of $p_i\sim \mathbf{1}$ there are isometries
 $s_i\in \mathcal{U}(\M)$ satisfying $s_i^*s_i=\mathbf{1}$ and $s_is_i^*=p_i$. One has
 $s_is_j^*s_ks_l^*=\delta_{jk}s_is_l^*$ and $(s_is_j^*)^*=s_js_i^*$.
 Hence  $V_{ij}=\{s_is_j^*\}$ is the self-adjoint system
 of the  $I\times I$ matrix units, i.e. it is the inverse  semigroup. \\
For $N>1$ the $\mathcal{V}(\{p_i\}_{i\in I})$ is an inverse subsemigroup
$\mathcal{V}_N\subset S_N$ of the Cuntz inverse semigroup $S_N$.   The
\emph{Cuntz inverse semigroup}  $S_N$    consists of  elements
$$s_{i_1}...s_{i_k}s_{j_1}^*...s_{j_l}^*$$ where $i_1,...,i_k,j_1,...,j_l\in I$, e.g.
see \cite{pater}. Denoting $k$-tuples by $\alpha:=(i_1,...,i_k)$ and the isometry
$s_{i_1}...s_{i_k}$ by $s_\alpha$ we find that the product and the involution in
Cuntz inverse semigroup $S_N$ are given by \be\begin{array}{l}
(s_\alpha s_{\beta\gamma}^*)(s_\beta s_\delta^*)=s_\alpha s_{\gamma\delta}^*\\
(s_\alpha s_{\beta}^*)(s_{\beta\gamma} s_\delta^*)=s_{\alpha\gamma}
s_{\delta}^*\end{array}\ee and \be(s_\alpha s_\beta^*)^*=s_\beta s_\alpha^*\ee
respectively, see \cite{renault}. The inverse semigroup $S_N$ generates the
$C^*$-subalgebra of $\M$ isomorphic to Cuntz algebra $\mathcal{O}_N$ \cite{Cuntz}.
 Note that for $N=1$ one obtains Toeplitz inverse semigroup.

 \begin{prop}\label{C}
 $W^*$-algebra $\M$ is properly infinite if and only if  it contains
 a Cuntz inverse semigroup $S_N$ such that $\mathbf{1}\in S_N$, where $N\in \mathbb{N}$ or $N=\infty$.
  \end{prop}
  \prf{ Summing up the facts presented above we find that any properly infinite $W^*$-algebra $\M$ contains a Cuntz inverse semigroup $S_N$ and unit of $\M$ belongs to $S_N$.

   Let us assume that $S_N\subset\M$ and $\mathbf{1}\in S_N$. Then from (\ref{cuntz}) we have
   \be\label{cuntz2} z=\sum_{i\in I}p_iz\qquad {\rm and}\qquad p_iz\sim z\ee for any central projection $z$ of $\M$. From (\ref{cuntz2}) follows that any central projection $z$ is equivalent to its not trivial subprojection. This means that $\M$ is properly infinite.
}

\bigskip

We recall that  a \emph{Clifford inverse semigroup} is an inverse
semigroup $\mathcal{S}$ whose  idempotents are central, see \cite{howie}.

\begin{prop}\label{centr}
Let $E\subset \mathcal{L}(\M)$ be a semilattice. We define \be\label{centr}
\mathcal{S}:=\bigcup_{p\in E} Z(p\M p)\cap \mathcal{U}(\M), \ee where $Z(p\M
p)$ is the center of the $W^*$-subalgebra $p\M p\subset\M$. Then
$\mathcal{S}$ is an inverse semigroup if $\mathcal{S}\subset \mathcal{U}(E)$.
Additionally  $\mathcal{S}$  is a Clifford inverse semigroup.
\end{prop}
\prf{Assuming $\mathcal{S}\subset  \mathcal{U}(E)$ we obtain from (\ref{centr})
that for  $u,v\in\mathcal{S}$ there are $p, q\in E$ such that $p=u^*u=uu^*$ and
$q=v^*v=vv^*$. From
$$uv(uv)^*uv=uvv^*u^*uv=uu^*uvv^*v=uv$$
and from
$$uv=upqv=(pqupq)(pqvpq)$$
it follows that $uv$ is a partial isometry and $uv\in pq\M pq$. Since $u\in Z(p\M
p)$, $v\in Z(q\M q)$ and $p\M p\cap q\M q=pq\M pq$ we have $uv\in Z(pq\M
pq)$. The above shows that $\mathcal{S}$ is an inverse semigroup. In order to
prove $qu=uq$ for any $u\in \mathcal{S}$ and any $q\in E$ we observe that
$u\in Z(p\M p)\cap \mathcal{U}(\M)$ for some $p\in E$. Since $pq\in Z(p\M
p)\cap \mathcal{U}(\M)$ one obtains $qu=qpu=uqp=uq$. }
\begin{rem}
The condition $\mathcal{S}\subset  \mathcal{U}(E)$ from Proposition \ref{centr}
is fulfilled, for example, if from $p\in E$ and $q\leqslant p$ follows $q\in E$.
\end{rem}

Now we will describe the inverse semigroups $\mathcal{S}\subset
\mathcal{\U}(\M)$ generated by a single partial isometry $u\in\mathcal{\U}(\M)$.
According  to \cite{howie} we will call such an inverse semigroup a
\emph{monogenic inverse semigroup}. By definition, see \cite{H-W}, the partial
isometry $u\in\M$ is a \emph{power partial isometry} if  $u^k$ is a partial
isometry for all $k\in \mathbb{N}$.

We conclude from  Remark \ref{rem1} that \be\label{monog1} up_{k+1}=p_k u
\quad {\rm{and}}\quad u^*p_k=p_{k+1}u^*\ee \be\label{monog2} q_{l+1}u=uq_l
\quad {\rm{and}}\quad q_l u^*=u^* q_{l+1}\ee where \be\label{monog3}
p_k:=u^{*k} u^k, \quad q_l:=u^lu^{*l}.\ee

\begin{lem}
Projections $p_k$ and $q_l$ defined in (\ref{monog3}) have the following
properties: \be\label{monog4} p_kp_l=p_{max\{k,l\}}\ee \be\label{monog5}
g_kq_l=q_{max\{k,l\}}\ee \be\label{monog6} p_kq_l=q_lp_k.\ee\end{lem}
 \prf{
Assuming $k>l$ we obtain \be\label{monog7}
p_k=u^{*k}u^k=u^{*l}p_{k-l}u^l=p_lp_k\ee where the last equality in
(\ref{monog7}) follows from (\ref{monog1}). Similarly
we prove that $q_k=q_lq_k$.\\
Since $v^ku^l=u^{k+l}$ is a partial isometry the projections $p_k=u^{*k} u^k$
and $q_l=u^l u^{*l}$ commute, see Remark \ref{rem1}. This proves
(\ref{monog6}). }

\begin{thm}\label{monogthm}
The partial isometry $u\in\mathcal{\U}(\M)$ generates a monogenic inverse
semigroup $\mathcal{S}_{\langle u\rangle}$ if and only if $u$ is a power partial isometry.\\
Every partial isometry from $\mathcal{ S}_{\langle u\rangle}$ can be expressed
in the form \be\label{monog8} p_kq_lu^m \quad {\rm{or}} \quad p_kq_lu^{*m}\ee
where $k, l, m\in \mathbb{N}$, while $u^0=1$
and $u^{*0}=1$. In particular the set of idempotents of $\mathcal{ S}_{\langle
u\rangle}$ consists of projections $p_kq_l$.
\end{thm}
\prf{If $\mathcal{S}_{\langle u\rangle}\subset \mathcal{\U}(\M)$ is a monogenic
inverse semigroup then  $u^k\in \mathcal{S}_{\langle u\rangle}$. So it is a
partial
isometry. So, $u$ is power partial isometry.\\
Now let us assume that for any $k\in \mathbb{N}$ element $u^k$ is a partial
isometry. Then the relations  (\ref{monog1}),  (\ref{monog2}), (\ref{monog4}),
(\ref{monog5}) and (\ref{monog6}) are valid. Using these relations one
transforms the arbitrary element \be\label{monog9}
u^{k_1}u^{*l_1}u^{k_2}u^{*l_2}...u^{k_N}u^{*l_N} \ee to the product of three
elements \be\label{monog10}  p_kq_lu^m \quad {\rm{or}} \quad
p_kq_lu^{*m}\ee which satisfy
$$p_kq_lu^m(p_kq_lu^m)^*=p_kq_lq_m=p_kq_{max\{l,m\}}$$
or $$p_kq_lu^{*m}(p_kq_lu^{*m})^*=p_kq_lp_m=q_lp_{max\{k,m\}}.$$

So, they are partial isometries. In such a way we show that the arbitrary
element (\ref{monog9})  generated by $u$ is a partial isometry of the form
(\ref{monog8}) and it is an idempotent if and only if  $m=0$. }

\begin{cor}\label{mon}
Each element of $\mathcal{S}_{\langle u\rangle}$  can be presented in the form
\be\label{monog11} u^ku^{*l}u^m\ee where $0\leqslant k\leqslant l$, $0\leqslant
m \leqslant l$ and $l>0$.
\end{cor}
 \prf{One obtains (\ref{monog11}) from (\ref{monog8})
using  (\ref{monog1}),  (\ref{monog2}), (\ref{monog4}), (\ref{monog5}) and
(\ref{monog6}).}

Since an arbitrary monogenic inverse semigroup $\mathcal{S}$ is isomorphic to
some $\mathcal{S}_{\langle u\rangle}$ the statement of Corollary \ref{mon}  is
valid also for each monogenic inverse semigroup. We thus proved a statement known as the Gluskin theorem \cite{Gluskin}.

\begin{example}
If  the generator $u$ of $\mathcal{S}_{\langle u\rangle}$ is an isometry
(co-isometry), i.e. $u^*u=\mathbf{1}$ ($uu^*=\mathbf{1}$), then the monogenic inverse semigroup
$\mathcal{S}_{\langle u\rangle}$ is the Toeplitz inverse semigroup. In this case
any element of $\mathcal{S}_{\langle u\rangle}$ can be written in the form
$q_lu^m$  ($p_ku^m$).
\end{example}

Other examples of inverse semigroups are given in the following
\begin{prop}
A projection $p\in \M$ and a unitary element $u\in\M$ generate an inverse
semigroup $\mathcal{S}_{\langle p,u\rangle}\subset \M$  if and only if
\be\label{com} [p,u^kpu^{*k}]=0 \ee for $k\in \N$.
\end{prop}
\prf  {Using for any $k\in\N$ the following notation
$$u^k:=\left\{\begin{array}{lcl} u^k&{\rm for}&k\in \No\\
u^{*k}&{\rm for}&-k\in \N  \end{array}\right.$$
we observe that any element $x\in\M$ generated by the products of $p$ and $u$ has the following form
\be\label{x}
x=u^{l_1}pu^{l_2}pu^{l_3}\ ...\ pu^{l_{N-1}}pu^{l_{N}},
\ee
where $l_1,\ l_2\,...,l_N\in \mathbb{Z}$. Let us assume (\ref{com}). Then we have \be\label{com2}
[u^kpu^{-k},u^lpu^{-l}]=0
\ee
for any $k,l\in \mathbb{Z}$. From (\ref{com2}) we obtain that $x^2=x=x^*\ $ if and only if $\ l_1+l_2+...+l_N=0$. Thus $xx^*$ and $x^*x$ are idempotents for any choice of $l_1,\ l_2\,...,l_N\in \mathbb{Z}$. Hence elements of $\M$ given by (\ref{x}) generate the inverse semigroup which we denote by $\mathcal{S}_{\langle p,u\rangle}$. \\
If $\mathcal{S}_{\langle p,u\rangle}\subset \M$ is the inverse semigroup then
the elements $p$ and $u^k p u^{*k}$ commute since they are idempotents of
$\mathcal{S}_{\langle p,u\rangle}$. }

\begin{cor}
The orthogonal projections $p, 1-p \in \M$ and a unitary element $u\in \M$
generate an inverse semigroup $\mathcal{S}_{\langle p,1-p,u\rangle }$ if and
only if the condition (\ref{com}) is fulfilled.
\end{cor}
\prf{ The set $\mathcal{S}_{\langle 1-p,u\rangle }$ is an inverse semigroup if
and only if $p$ and $u$ satisfy the condition (\ref{com}). From this and from
$[p,1-p]=0$ we have that $\mathcal{S}_{\langle p,1-p,u \rangle}$ is the inverse
semigroup if and only if  (\ref{com}) is fulfilled. }

Finally let us give an example of inverse semigroup with a
remarkable significance  in quantum physics. To this end let us consider the
sequence of bounded operators  $a_1,a_2,..., a_N \in L^\infty(\H)$, where
$N\in \N$ or $N=\infty$, acting in the separable Hilbert space $\H$ and satisfying
the canonical anticommutation
relations (CAR) \be\label{CAR}\begin{array}{l} a_ia_j^*+a_j^*a_i=\delta_{ij}1\\
a_ia_j+a_ja_i=0\end{array}.\ee It follows from (\ref{CAR}) that the
annihilation operators $a_1,a_2,..., a_N$ and their adjoints
$a_1^*,a_2^*,..., a_N^*$ (i.e. the creation operators) are partial
isometries. The projections $P_i:=a_ia_i^*$, $\ i=1,2,...,N$, and
the complementary projections $Q_j:=1-P_j=a_j^*a_j$, $\ j=1,2,...,n$
commute. So, they generate a semilattice $E\subset \mathcal{L}(\M)$. \\
Any element $x\in \M$ generated by the products of the annihilation and creation operators has the following form
\be\label{elemCAR} x=cP_\alpha Q_\beta a^*_\gamma a_\delta,\ee where  $c\in \{-1,1\}$ and the products
$P_\alpha:=P_{\alpha_1}P_{\alpha_2}...P_{\alpha_k}$, $\ Q_\beta:=Q_{\beta_1}Q_{\beta_2}...Q_{\beta_l}$, $\
a^*_\gamma:=a^*_{\gamma_1}a^*_{\gamma_2}...a^*_{\gamma_m}$, $\ a_\delta:=a_{\delta_1}a_{\delta_2}...a_{\delta_n}$
 are enumerated by the increasing sequences of  indices from $\{1,2,...,N\}$ such that for
$A,B\in\{\alpha,\beta,\gamma,\delta\},\ A\neq B$ implies that $A\cap B=\emptyset$.

The product of the two such elements is given by \be
xx'=c''P_{\alpha''} Q_{\beta''} a^*_{\gamma''} a_{\delta''}, \ee
where $c''=cc'(-1)^{m-i+j-p+r}$ for $\gamma_i=\delta_j'$, $\
\delta_p=\gamma_r'$\\
$\alpha''=\alpha\cup\alpha'\cup(\delta\cap\gamma')\setminus(\delta'\cup\gamma)$\\
$\beta''=\beta\cup\beta'\cup(\delta'\cap\gamma)\setminus(\delta\cup\gamma')$\\
$\gamma''=\gamma\cup\gamma'\setminus(\delta\cup\delta')$\\
$\delta''=\delta\cup\delta'\setminus(\gamma\cup\gamma')$ \\
if $\alpha\cap\gamma'=\emptyset, \ \beta\cap\delta'=\emptyset, \ \alpha\cap\beta'=\emptyset, \
\beta\cap\alpha'=\emptyset,\ \gamma\cap\gamma'=\emptyset,\ \delta\cap\delta'=\emptyset$. In the opposite case one has
$xx'=0$.

 In particular the products of the conjugated elements
take the following forms:

\be \begin{array}{l}xx^*=cP_\alpha Q_\beta a^*_\gamma a_\delta
a^*_\delta a_\gamma Q_\beta P_\alpha
c=P_{\alpha\cup\beta}Q_{\beta\cup\gamma}\\
\\
x^*x=a^*_\delta a_\gamma Q_\beta P_\alpha P_\alpha Q_\beta
a^*_\gamma
a_\delta=P_{\alpha\cup\gamma}Q_{\beta\cup\delta}\end{array}. \ee

Concluding, we see that subset $\mathcal{S}\subset L^\infty(\H)$ consisting of
elements (\ref{elemCAR}) is an inverse semigroup. Let us call it the CAR
inverse semigroup.

\bigskip

Ending this section we go back to the Proposition \ref{propLambda} and describe a faithful representation of the inverse semigroup of local bisections $\mathfrak{B}_{loc}(\U(\M)*_{}\M_*^+)$ of the action groupoid  $\U(\M)*_{}\M_*^+$ in the $W^*$-algebra $\ol{\rho}(\M)'$. To this end let us note the local bisection $\sigma\in\mathfrak{B}_{loc}(\U(\M)*_{}\M_*^+)$ is a map
\be\label{locsemi}\sigma_{\Omega}:\Omega\ \ni\ \omega\mapsto\sigma_\Omega(\omega)=(\tilde{\sigma}(\omega),\omega)\ \in \ \U(\M)*_{}\M_*^+\ee
of $\Omega\subset\M_*^+$ which satisfies
$$\Ss(\sigma(\omega))=\omega,$$
$$\Tt(\sigma(\omega))=I_{*\tilde{\sigma}(\omega)}\omega.$$ Using groupoids monomorphism (\ref{Lambda}) we define the representation \be\label{locrep}\phi:\mathfrak{B}_{loc}(\U(\M)*_{}\M_*^+)\to\ol{\rho}(\M)'\ee
of $\mathfrak{B}_{loc}(\U(\M)*_{}\M_*^+)$ in $\ol{\rho}(\M)'$ as follows
\be\label{locrep2}\phi(\sigma_\Omega):=\sum_{\omega\in\Omega} (\iota\circ\phi)(\tilde{\sigma}(\omega),\omega).\ee
The idempotent corresponding to $\phi(\sigma_\Omega)$ is given by
$$\phi(\sigma_\Omega)^*\phi(\sigma_\Omega)=\sum_{\omega\in\Omega}id_\omega$$
where $id_\omega$ is the identity in $\ol{\mathbb{E}}_\omega\subset L^2\Gamma(\mathbb{E},\M_*^+).$

\section {Topologies and Banach manifold structure of groupoids $\G(\M)$ and $\U(\M)$}

There are the following locally convex topologies considered on  the $W^*$-algebra $\M$: the uniform topology, the
Arens-Mackey topology $\tau(\M,\M_*)$, the strong $^*$-topology
$s^*(\M,\M_*)$, the strong topology $s(\M,\M_*)$, the $\sigma$-weak topology
$\sigma(\M,\M_*)$, see e. g.\cite{sakai}. All these topologies define corresponding topologies on the groupoids $\G(\M)$ and $\U(\M)$. Hence, the natural question  arises for which of the topologies listed above the groupoids are  topological groupoids.

Let us start  from $\G(\M)$.
\begin{prop}
For a infinite-dimmentional $W^*$-algebra $\M$ the groupoid $\G(\M)$ is not a topological groupoid with respect to any  topology of $\M$ mentioned above.
\end{prop}
\prf{Let us take $p\in \mathcal{L}(\M)$ and define $x_n\in \G(\M)$ by
$$x_n=p+\frac{1}{n}(1-p),\quad n\in\mathbb{N}.$$
One has $$\Ss(x_n)=\Tt(x_n)=1\quad {\rm and}\quad \Ss(p)=\Tt(p)=p.$$
Since the uniform limit of $x_n$  is $$p=\lim_{n\to\infty} \ x_n,$$
we see that source and target maps of $\G(\M)$ are not continuous. Thus we obtain that $\G(\M)$ is not a topological one. Note that the above consideration does not depend on the choice of topology on $\M$.
}
The case of the groupoid $\U(\M)$ is much better than that of $\G(\M)$. Let us begin our considerations from the uniform topology. Since all algebraic
operations in $\M$ and the $^*$-involution are uniformly continuous  and groupoid maps are expressed by these operations we conclude that  the
groupoid $\mathcal{U}(\M)$ is a topological groupoid  with
respect to the uniform topology. Let us remark also that $\U(\M)$ is uniformly closed in $\M$ and $\mathcal{L}(\M)$ is uniformly
closed in $\mathcal{U}(\M)$. Note also that the  set $\mathcal{U}(\M)^{(2)}=(\Ss\times
\Tt)^{-1}(\{(p,p):\;p\in \mathcal{L}(\M)\})$ is closed in $\mathcal{U}(\M)\times
\mathcal{U}(\M)$.

The groupoid of partial isometries $\mathcal{U}(\M)$ is not topological with
respect to the $\sigma(\M,\M_*)$-topology (the weak $^*$-topology) and with
respect to the $s(\M,\M_*)$-topology (the strong topology). The reason is that the
product map (\ref{prod})
is not continuous with respect to $\sigma(\M,\M_*)$-topology and the involution (\ref{inverse}) is not continuous with respect to the $s(\M,\M_*)$-topology.

The Arens-Mackey topology $\tau(\M,\M_*)$ coincides with the $s^*$-strong
topology $s^*(\M,\M_*)$ on the bounded parts of $\M$, see \cite{sakai}. So
both of them induce on $\mathcal{U}(\M)$ the same topology. So, without
loss of  generality we can restrict our consideration to the
$s^*(\mathcal{U}(\M))$-topology  of $\ \mathcal{U}(\M)$.

Let us take the closed unit ball $\B=\{x\in \M:\; \norm x\leqslant 1\}$ in
$\M$. The product map $\B\times \B\ \ni\ (x,y)\ \mapsto\ xy\ \in\ \B $ restricted
to $\B$  as well as the $^*$-involution are continuous  with respect to $s^*(\B,\M_*)$-topology. From the
above we conclude:

\begin{prop}
The  groupoid $\mathcal{U}(\M)$ of partial isometries is a topological
groupoid  with respect to the
$s^*(\mathcal{U}(\M),\M_*)$-topology.
\end{prop}
 Let us define on $\M_*\cong\{(p,\omega)\in
\mathcal{L}(\M)\times \M_*;\ \ p=r_*(\omega)\}$ (respectively
 $\M_*\cong\{(p,\omega)\in \mathcal{L}(\M)\times \M_*;\ \ p=l_*(\omega)\}$) the
topology $\mathfrak{T}_{\M_*}$  as the topology inherited from the product topology
of $\mathcal{L}(\M)\times \M_*$.  The moment map $r_*:\M_*\rightarrow \mathcal{L}(\M)$
(respectively  $l_*:\M_*\rightarrow \mathcal{L}(\M)$) is continuous with respect to
$\mathfrak{T}_{\M_*}$. Since the topology $\mathfrak{T}_{\M_*}$  of $\M_*$ is stronger
than the uniform topology of $\M_*$ the action (\ref{rag}) (respectively
(\ref{lag})) is also continuous with respect to  $\mathfrak{T}_{\M_*}$.

Let us define the set \be\label{P}\P:=\{\omega\in\M_*:\quad l_*(\omega)=r_*(\omega)\}.\ee
 We conclude from the Proposition (\ref{pred}) that subsets
 $\M_*^+\subset \M_*^h \subset \P(\M_*)\subset\M$ of positive normal functionals,
 selfadjoint functionals and $\P(\M)$ are  invariant with respect to the predual inner action
 $I_*:\mathcal{U}(\M)\times \M_*\to \M_*$. The groupoid  $\mathcal{U}(\M)$ acts on $\M_*^h$,
 $\M_*^+$ and $\P(\M)$ the actions are continuous with respect to their $\mathfrak{T}_{\M_*}$-topology.
Since
$$s_*(I_{*u} \omega)=s_*(u\omega u^*)=uu^*=\Tt(u)$$
$$
I_{*u}(I_{*v}\omega)=I_{*u}(v\omega v^*)=uv\omega v^*u^*=uv\omega (uv)^*=I_{*uv}\omega$$
$$I_{*\varepsilon(s_*(\omega))}\omega=I_{*s_*(\omega)}\omega=u^*u\omega u^*u=\omega$$
we see that the groupoid $\mathcal{U}(\M)$ acts on $\P(\M_*)$ in the
continuous way with respect to $\mathfrak{T}_{\M_*}$ topology of
$\P(\M_*)$.

Summarizing the above considerations  and applying the construction
presented in the Appendix  we have the following:

\begin{thm}\label{4.3}
\ben[(i)]
\item
The groupoids $\ \mathcal{U}(\M)*_l\M\ $, $\ \mathcal{U}(\M)*_r\M\ $,
$\mathcal{U}(\M)* \mathcal{J}(\M)\ $, $\ \mathcal{U}(\M)*\M^h$ and $\
\mathcal{U}(\M)*\M^+\ $ are topological groupoids with respect to the
relative topology inherited from the product uniform topology of $\
\mathcal{U}(\M)\times \M\ $.
\item
The groupoids $\ \mathcal{U}(\M)*_{l_*}\M_*$ , $\ \mathcal{U}(\M)*_{r_*}\M_*\ $,
$\mathcal{U}(\M)* \P(\M_*)\ $, $\ \mathcal{U}(\M)*\M_*^+\ $ and $\ \mathcal{U}(\M)
*\M_*^h$ are topological groupoids with respect to the  relative topology
inherited from the product uniform topology of $\ \mathcal{U}(\M)\times \M_*$.
\item The groupoids listed above cover the groupoid $\mathcal{U}(\M)$.
\een
\end{thm}

Now we show that the groupoids $\G(\M)$ and $\U(\M)$ have canonically defined structure of complex and real Banach manifold respectively.
Let us begin from definition of complex Banach manifold structure on the lattice $\mathcal{L}(\M)$ of projections of $W^*$-algebra $\M$. For this reason for any $p\in \mathcal{L}(\M)$ by $\Pi_p\subset \mathcal{L}(\M)$ we denote the subset of projections $q\in \mathcal{L}(\M)$ such that
\be\label{complem} q\wedge(1-p)=0\quad {\rm and} \quad q\vee(1-p)=1,\ee
where $"\wedge"$ and $"\vee"$ are joint- and meet-operations  on the projections in the lattice $\mathcal{L}(\M)$. Since for any pair of projections $e,f\in\mathcal{L}(\M)$ one has
$$(e\vee f)-e\sim f-(e\wedge f),$$
see \cite{takesaki}, taking $e=1-p$ and $f=q$ we obtain that $q\sim p$. So we have $\Pi_p\subset\mathcal{O}_p$ and thus $\Pi_{p'}\cap\Pi_p=\emptyset$ if $p'\not\sim p$. The inverse statement, i.e. that $p'\sim p$ implies $\Pi_{p'}\cap\Pi_p\not=\emptyset$ is not true in general case. For example for infinite $W^*$-algebra we can take $p\not=1$ such that $p\sim 1$. Then $\Pi_1=\{1\}$ and thus $\Pi_1\cap\Pi_p=\emptyset$.\\
Condition (\ref{complem}) is equivalent to existence of Banach splitting
\be\label{dec2}  \M=q\M\oplus(1-p)\M\ee
of $\M$ on the right $W^*$-ideals. Using  (\ref{dec2}) we decompose
\be\label{dekompp} p=x-y\ee
 the projection $p$ on two elements $x\in q\M p$ and $y\in (1-p)\M p$. In such a way we define the map  $\varphi_p:\Pi_p\ \tilde{\to}\ (1-p)\M p$  by
 \be \varphi_p(q):=y.\ee
Let us show that $\varphi_p$ is a bijection of $\Pi_p$ on the Banach space $(1-p)\M p$. To this end for any $y\in (1-p)\M p$ we define $x$ by equality (\ref{dekompp}) and note that
 \be\label{propq} p=px,\quad xp=x\quad {\rm and}\quad x^2=x.\ee
 Thus the left multiplication maps $L_p$ and $L_x$ on $\M$ satisfy
 \be\label{L} L_p=L_p\circ L_x,\quad L_x=L_x\circ L_p\quad {\rm and}\quad L_x\circ L_x=L_x\ee
 and
 \be\label{kerLx} Ker\ L_x=(1-x)\M=(1-p)\M= Ker\ L_p.\ee
 From (\ref{kerLx}) we have
  \be\label{dec1} \M=x\M\oplus(1-p)\M,\ee
  where $x\M$ is right ideal of $W^*$-algebra generated by $x\in \M$. Let us also note that
  \be\label{LL} L_x:p\M\ \tilde{\to}\  x\M \quad {\rm and}\quad L_p:x\M\ \tilde{\to}\ p\M\ee
  are mutually inverse isomorphisms of the corresponding right $W^*$-ideals.

 The left support $l(x)$ of $x\in\M$ is the identity in $W^*$-subalgebra $x\M\cap(x\M)^*$. Thus $l(x)\in x\M$. This shows that $x\M=l(x)\M$ and
  \be\label{decomp3} \M=l(x)\M\oplus(1-p)\M,\ee
  i.e. $l(x)\in \Pi_p$. In such a way we  prove that $\varphi_p$ has the inverse defined by
  \be\label{phiinverse} \varphi^{-1}_p(y):=l(p+y).\ee
  \begin{prop}
  If $x\in q\M p$ is defined by the decomposition (\ref{dekompp}) then $x\in\G(\M)$ and
$\Ss(x)=p$ and $ \Tt(x)=q$. So one has section
$\sigma_p: \Pi_p\to \Tt^{-1}(\Pi_p)\subset \G(\M)$ defined by
\be\label{sigma} \sigma_p(q):=x.\ee
  \end{prop}
  \prf{The above follows from (\ref{LL}) and from $x\M=l(x)\M=q\M$.  }

  We conclude from the above that maps $\varphi_p: \Pi_p\to (1-p)\M p$, where $p\in\mathcal{L}(\M)$, define a canonical atlas on $\mathcal{L}(\M)$. Recall that $\bigcup_{p\in\mathcal{L}(\M)}\Pi_p=\mathcal{L}(\M)$. This atlas is modeled by the family of Banach spaces $(1-p)\M p$. If projections $q$ and $p$ are equivalent, i.e. if there exists $x\in \G(\M)$ such that $p=\Ss(x)$ and $p'=\Tt(x)$, then the Banach spaces $(1-p)\M p$ and $(1-p')\M p'$ are isomorphic.

  Now we find the explicite formulae for the transitions maps
  \be\label{trans} \varphi_p\circ\varphi^{-1}_{p'}: \varphi_{p'}(\Pi_p\cap\Pi_{p'})\to \varphi_{p}(\Pi_p\cap\Pi_{p'})\ee
  in the case when $\Pi_p\cap\Pi_{p'}\not= \emptyset$. For this reason let us take for $q\in \Pi_p\cap\Pi_{p'}$ the following splittings
  \be\label{trans1}\begin{array}{l} \M=q\M\oplus(1-p)\M=p\M\oplus(1-p)\M \\
  \M=q\M\oplus(1-p')\M=p'\M\oplus(1-p')\M.\end{array}\ee
 The splittings (\ref{trans1}) lead to  the corresponding decompositions of $p$ and $p'$
  \be\label{trans2}\begin{array}{lll} p=x-y& \qquad & p=a+b\\
  p'=x'-y'&\qquad & 1-p=c+d\end{array}\ee
  where $x\in q\M p$, $\ y\in (1-p)\M p$, $\ x'\in q\M p'$, $\ y'\in (1-p')\M p'$, $\ a\in p'\M p$, $\ b\in (1-p')\M p$, $\ c\in p'\M (1-p)\ $ and  $\ d\in (1-p')\M (1-p)$. Combining equations from (\ref{trans2}) we obtain
  \be\label{trans3} q=\iota(x')+y' \iota(x')\ee
  \be\label{trans4} q=(a+c y)\iota(x)+(b+d y) \iota(x).\ee
  Comparing (\ref{trans3}) and (\ref{trans4}) we find that
  \be\label{trans5} \iota(x')=(a+c y)\iota(x)\ee
  \be\label{trans6} y' \iota(x')=(b+dy)\iota(x).\ee
  After substitution (\ref{trans5}) into (\ref{trans6}) and noting that $\Tt(a+cy)\leqslant p'$ we get
  \be\label{trans7} y'=(\varphi_{p'}\circ\varphi^{-1}_{p})(y)=(b+dy)\iota(a+cy).\ee
  All operations involved in the right-hand-side of equality (\ref{trans7}) are smooth. Thus we conclude that the atlas $\left(\Pi_p,\ \varphi_p\right)$, $p\in \mathcal{L}(\M)$ defines on $\mathcal{L}(\M)$ the structure of a  complex Banach manifold of type $\mathfrak{G}$, see \cite{bou}, where $\mathfrak{G}$ is the set of Banach spaces $(1-{p})\M{p}$ indexed by elements $p\in \mathcal{L}(\M)$. See also \cite{BG} for the investigation of the infinite-dimensional  Grassmannian  as a homogeneous spaces of the Banach-Lie group $U(\M)$. Note, that when $\M$ is finite $W^*$-algebra then the orbits of the inner action of the groupoid $\U(\M)$  and the orbits of the inner action of unitary group $U(\M)$ on the lattice $\mathcal{L}(\M)$ coincide.

  \bigskip

  Now let us introduce a structure of the Banach smooth manifold on the groupoid $\G(\M)$.

 For this reason taking $p,\tilde{p}\in \mathcal{L}(\M)$ we define the covering of $\G(\M)$ by subsets:
 \be\label{OMega} \Omega_{\tilde{p}p}:=\Tt^{-1}(\Pi_{\tilde{p}})\cap\Ss^{-1}(\Pi_p).\ee
 Let us note here that $\Omega_{\tilde{p}p}\not=\emptyset$ if and only if $\tilde{p}\sim p$. Note also that the set $\Omega_{pp}$ is a subgroupoid of $\G(\M)$. If $\Omega_{\tilde{p}p}\not=\emptyset$ then one has the one-to-one map
 \be\psi_{\tilde{p}p}:\Omega_{\tilde{p}p}\to (1-\tilde{p})\M\tilde{p}\oplus\tilde{p}\M p\oplus(1-p)\M p\ee
 of $\Omega_{\tilde{p}p}$ on an open subset of the direct sum of the Banach subspaces of the $W^*$-algebra $\M$. This map we define by
 \be\label{psipp}\psi_{\tilde{p}p}(x):=\left(\varphi_{\tilde{p}}(\Tt(x)),\iota(\sigma_{\tilde{p}}(\Tt(x)))x \sigma_{{p}}(\Ss(x)),\varphi_p(\Ss(x)) \right),\ee
 where $\sigma_p(q)\in q\M p$ and $\varphi_p(q)\in (1-p)\M p$ are obtained from the decomposition
\be\label{dec3}p=\sigma_p(q)+\varphi_p(q)\ee
 of $p$ with respect to (\ref{dec2}). Recall that $\sigma_p:\pi_p\to\Tt^{-1}(\Pi_p)$ is a section defined in (\ref{sigma}).\\
 The map $\psi_{\tilde{p}p}^{-1}:\psi_{\tilde{p}p}(\Omega_{\tilde{p}p})\to \Omega_{\tilde{p}p}$ inverse to (\ref{psipp}) looks as follows
 \be\label{psiinv} \psi_{\tilde{p}p}^{-1}(\tilde{y},z,y):=\sigma_{\tilde{p}}(\tilde{q})z\iota(\sigma_p(q))=(\tilde{p}+\tilde{y})z\iota(p+y)\ee
 where $\tilde{q}=l(\tilde{p}+\tilde{y})$ and $q=l(p+y)$ are left supports of $\tilde{p}+\tilde{y}$ and $p+y$ respectively.
 The transition  maps
 $$\psi_{\tilde{p'}p'}\circ\psi_{\tilde{p}p}^{-1}:\psi_{\tilde{p}p}(\Omega_{\tilde{p'}p'}\cap\Omega_{\tilde{p}p})\to
 \psi_{\tilde{p'}p'}(\Omega_{\tilde{p'}p'}\cap\Omega_{\tilde{p}p})$$
 for $(\tilde{y},z,y)\in \psi_{\tilde{p}p}(\Omega_{\tilde{p'}p'}\cap\Omega_{\tilde{p}p})$ are given by
 \be
 (\psi_{\tilde{p'}p'}\circ\psi_{\tilde{p}p}^{-1})(\tilde{y},z,y):=(\tilde{y'},z',y'),
 \ee
 where
 \be\label{tR1}
 \tilde{y'}=(\varphi_{\tilde{p'}}\circ\varphi^{-1}_{\tilde{p}})(\tilde{y})=(\tilde{b}+\tilde{d}\tilde{y})\iota(\tilde{a}+\tilde{c}\tilde{y}) \ee
 \be\label{tR2} y'=(\varphi_{p'}\circ\varphi^{-1}_{p})(y)=(b+dy)\iota(a+cy)\ee
 and
 \be\label{tR3} z'=\iota(\tilde{p'}+\tilde{y'})(\tilde{p}+\tilde{y})z\iota(p+y)(p'+y').\ee
 We note that all maps in (\ref{tR1}), (\ref{tR2}), (\ref{tR3}) are smooth. Thus we conclude that  \be\label{atlas}\left(\Omega_{\tilde{p}p},\psi_{\tilde{p}p}:\Omega_{\tilde{p}p}\to (1-\tilde{p})\M\tilde{p}\oplus\tilde{p}\M p\oplus(1-p)\M p\right),\ee
 where $(p,\tilde{p})\in \mathcal{L}(\M)\times \mathcal{L}(\M)$ are pairs of equivalent projections,
 form smooth atlas on the groupoid $\G(\M)$ in sense of \cite{bou}. The smooth (analytic) Banach manifold structure of $\G(\M)$ has type $\mathfrak{G}$, where $\mathfrak{G}$ is the set of Banach spaces $(1-\tilde{p})\M\tilde{p}\oplus\tilde{p}\M p\oplus(1-p)\M p$ indexed by the pair of equivalent elements of $\mathcal{L}(\M)$.

 \begin{thm}\label{BLgroupoidthm}
 The groupoid $\G(\M)$ is a Banach-Lie groupoid on base $\mathcal{L}(\M)$ with respect to the smooth (analytic) Banach manifold structure of type  $\mathfrak{G}$ defined by the atlas (\ref{atlas}).
 \end{thm}
 \prf{ We show that all groupois maps and the groupoid product are smooth (analytic) with respect to the considered Banach manifold structure.
 \ben[(i)]
 \item For source and target map we have
 \be\label{atlas1} (\varphi_p\circ\Ss\circ \psi_{\tilde{p}p}^{-1})(\tilde{y},z,y)=y,\ee
 \be\label{atlas2} (\varphi_{\tilde{p}}\circ\Tt\circ \psi_{\tilde{p}p}^{-1})(\tilde{y},z,y)=\tilde{y}.\ee
 We assumed in (\ref{atlas1}) and (\ref{atlas2}) that  $(\tilde{y},z,y)\in \psi_{\tilde{p}p}(\Omega_{\tilde{p}p})$, $\ \Ss(\psi_{\tilde{p}p}^{-1}(\tilde{y},z,y))\in \Pi_p$ and $\Tt(\psi_{\tilde{p}p}^{-1}(\tilde{y},z,y))\in \Pi_{\tilde{p}}$ respectively. We conclude from (\ref{atlas1}) and (\ref{atlas2}) that $\varphi_p\circ\Ss\circ \psi_{\tilde{p}p}^{-1}$ and $\varphi_{\tilde{p}}\circ\Tt\circ \psi_{\tilde{p}p}^{-1}$ are smooth (analytic) submersions.
 \item For identity section $\varepsilon:\mathcal{L}(\M)\to\G(\M)$ we have
 \be\label{atlas3}(\psi_{\tilde{p}p}\circ\varepsilon\circ\varphi_p^{-1})(y)=(\varphi_{\tilde{p}}\circ\varphi_p^{-1}(y),
 \iota(\sigma_{\tilde{p}}(\varphi^{-1}_p(y)))\sigma_{p}(\varphi^{-1}_p(y)),y),\ee
 where $y\in \varphi_p(\Pi_p)$. Since $\sigma_{\tilde{p}}:\Pi_{\tilde{p}}\to\Tt^{-1}(\Pi_{\tilde{p}})$ and $\sigma_{p}:\Pi_{p}\to\Tt^{-1}(\Pi_{p})$ are  smooth (analytic) sections we obtain that $\psi_{\tilde{p}p}\circ\varepsilon\circ\varphi_p^{-1}$ is smooth (analytic) map too.
 \item The inverse map $\iota:\G(\M)\to\G(\M)$ takes  $\Omega_{\tilde{p}p}$ onto $\Omega_{p\tilde{p}}$ and we have
 \be\label{atlas4}(\psi_{p\tilde{p}}\circ\iota\circ\psi_{\tilde{p}p}^{-1})(\tilde{y},z,y)=({y},\iota({z}),\tilde{y}).\ee
 Thus $\iota$ is a complex smooth (analytic) map.
 \item Let us take $x_1\in \Omega_{\tilde{p}_1p_1}$ and $x_2\in \Omega_{\tilde{p}_2p_2}$ such that $\Ss(x_1)=\Tt(x_2)\in \Pi_{\tilde{p}_2}\cap\Pi_{p_1}$. Assuming $\psi_{\tilde{p}_1p_1}(x_1)=(\tilde{y_1},z_1,y_1)$ and $\psi_{\tilde{p}_2p_2}(x_2)=(\tilde{y_2},z_2,y_2)$ we obtain that
     \be\label{atlas5}\psi_{\tilde{p}_1p_2}(\psi_{\tilde{p}_1p_1}^{-1}(\tilde{y_1},z_1,y_1)\psi_{\tilde{p}_2p_2}^{-1}
     (\tilde{y_2},z_2,y_2))=\ \ \ \ \ \ \ \ \ \ \ \ \ \ \ \ \ \ \ \ \ee
     $\ \ \ \ \ \ \ \ \ \ \ \ =(\tilde{y_1},z_1\iota(\sigma_{p_1}(\varphi^{-1}_{p_1}(y_1)))(\sigma_{\tilde{p}_2}(\varphi^{-1}_{\tilde{p}_2}(\tilde{y}_2)))z_2,y_2).$\\

 \een
 Summing up we conclude that $\G(\M)$ is a Banach-Lie groupoid. }

 In order to investigate the structure of real Banach manifold on $\U(\M)$ we recall that one can define $\U(\M)$ as the set of the fixed points of the groupoid automorphism $J:\G(\M)\to\G(\M)$, see (\ref{J}). Next expressing $J:\Omega_{\tilde p p}\to\Omega_{\tilde p p}$ in the coordinates
 $$\Omega_{\tilde{p}p}\ \ni\  x\longmapsto\psi_{\tilde{p}p}(x)=\left(\tilde y,z,y\right)\ \in\  (1-\tilde{p})\M\tilde{p}\oplus\tilde{p}\M p\oplus(1-p)\M p$$ we find that
 \be\label{mapa} \left(\psi_{\tilde{p}p}\circ J\circ \psi_{\tilde{p}p}^{-1}\right)\left(\tilde y,z,y\right)=\qquad\qquad\qquad\qquad\qquad\qquad\ee $$\qquad\qquad\qquad=\left(\tilde y,\iota(\sigma_{\tilde p}(\varphi^{-1}_{\tilde p}(\tilde y))^*\sigma_{\tilde p}(\varphi^{-1}_{\tilde p}(\tilde y)))\iota(z^*)\sigma_{ p}(\varphi^{-1}_{ p}( y))^*\sigma_{ p}(\varphi^{-1}_{ p}( y)),y\right),$$
 where $z\in \G(\M)_p^{\tilde p}\subset \tilde p\M p$. Note that  $\G(\M)_p^{\tilde p}$ is an open subset of the Banach subspaces $\tilde p\M p$. Since $J^2(x)=x$ for $x\in\U(\M)$ one has \be\label{DJ} (DJ(x))^2=\mathbf{1}\ee for $DJ(x):T_x\G(\M)\to T_x\G(\M)$.
  Thus one obtains a spliting of the tangent space \be T_x\G(\M)=T^+_x\G(\M)\oplus T^-_x\G(\M)\ee defined by the Banach space projections
 \be P_{\underline +}(x):=\frac{1}{2}\left(\mathbf{1}\underline + DJ(x)\right).\ee
 The Frech\'{e}t derivative $D\iota(z)$ of the inversion map
$$ \iota: \G(\M)_p^{\tilde p}\ni \ z\mapsto\ \iota(z)\in \G(\M)^p_{\tilde p}$$ at the point $z$ is given by \be D\iota(z)\vartriangle\! z=-\iota(z)\vartriangle\! z\ \iota(z), \ee
where $\vartriangle\! z\in \tilde p\M p$. Thus for $\vartriangle\!\tilde y\in (1-\tilde p)\M p$, $\vartriangle\! z\in \tilde p\M p$, $\vartriangle\! y\in (1-p)\M p$ we obtain
\be\label{DJ} D\left(\psi_{\tilde{p}p}\circ J\circ \psi_{\tilde{p}p}^{-1}\right)\left(\tilde y,z,y\right)(\vartriangle\!\tilde y,\vartriangle\! z,\vartriangle\! y)=\ee $$=\left(\vartriangle\!\tilde y,\ \ -\tilde g^{-2}(\tilde y)D\tilde g^2(\tilde y)\vartriangle\!\tilde y\tilde g^{-2}(\tilde y)\iota(z^*)g^2(y)-\tilde g^{-2}(\tilde y)\iota(z^*)\vartriangle\! z^*\iota(z^*)g^2(y)+\right.$$  $$\left. +\tilde g^{-2}(\tilde y)\iota(z^*)Dg^2(y)\vartriangle\! y,\ \ \vartriangle\! y\right),$$
where $$\tilde g^(\tilde y):=[\sigma_{\tilde p}(\varphi^{-1}_{\tilde p}(\tilde y))^*\sigma_{\tilde p}(\varphi^{-1}_{\tilde p}(\tilde y))]^{\frac{1}{2}}\quad
{\rm and}\quad g(y):=[\sigma_{ p}(\varphi^{-1}_{ p}( y))^*\sigma_{ p}(\varphi^{-1}_{ p}( y))]^{\frac{1}{2}}$$
are elements of $G(\tilde p\M \tilde p)$ and $G(p\M p)$ respectively.\\
If $x\in \Omega_{\tilde{p}p}\cap \U(\M)$ then from (\ref{mapa}) one has \be\label{Jz}\tilde g^2(\tilde y) z g^{-2}(y)=\iota(z^*).\ee It follows from (\ref{DJ}) that $\vartriangle\! z\in T_x^+\G(\M)$ if and only if it satisfies \be\label{deltaz} \vartriangle\! z=-\tilde g^{-2}(\tilde y)D\tilde g^2(\tilde y)\vartriangle\!\tilde y\tilde g^{-2}(\tilde y)\iota(z^*)g^2(y)+\ee $$-\tilde g^{-2}(\tilde y)\iota(z^*)\vartriangle\! z^*\iota(z^*)g^2(y)+ \tilde g^{-2}(\tilde y)\iota(z^*)Dg^2(y)\vartriangle\! y. $$ In order to simplify the equation (\ref{deltaz})  we introduce a new coordinate
\be\label{gug} u:=\tilde g(\tilde y)zg^{-1}(y)\ee for which the condition (\ref{Jz}) reduces to $u=\iota(u^*)$, i.e. \be u\in \tilde p\M p\cap\U(\M).\ee Substituting $z=\tilde g^{-1}(\tilde y)u g(y)$ and
\be \vartriangle\! z=-\tilde g^{-1}(\tilde y)D\tilde g(\tilde y)\tilde g^{-1}(\tilde y)\vartriangle\!\tilde y u g(y)+\ee $$+\tilde g^{-1}(\tilde y)\vartriangle\! u g(y)+ \tilde g^{-1}(\tilde y)u Dg(y)\vartriangle\! y$$
and \be \vartriangle\! z^*=Dg(y)\vartriangle\! y u^*\tilde g^{-1}(\tilde y)+g(y)\vartriangle\! u^*g^{-1}+\ee $$-g(y)u^*\tilde g^{-1}(\tilde y)D \tilde g(\tilde y)\vartriangle\! \tilde y g^{-1}(y)$$ into the equation (\ref{deltaz}) we obtain the equivalent equation \be u^*\vartriangle\! u+(u^*\vartriangle\! u)^*=0.\ee
 Thus for $x\in \Omega_{\tilde{p}p}\cap \U(\M)$ we have the isomorphisms of the real Banach spaces \be\label{T+} T^+_x\G(\M)\cong  (1-\tilde{p})\M\tilde{p}\oplus \mathfrak{H}_p \oplus(1-p)\M p\ee   where the real Banach space \be \mathfrak{H}_p=\{u^*\vartriangle\! u\in p\M p:\ \ u^*\vartriangle\! u+(u^*\vartriangle\! u)^*=0\}\ee  is  isomorphic with the space $ip\M^h p$ of the anti-hermitian elements of the subalgebra $p\M p$. Note here that fixing  $u_0\in\U(\M)_p^{\tilde p}\subset\G(\M)_p^{\tilde p}$ we obtain the bijection \be\U(\M)_p^{\tilde p} \ni u\ \tilde\longmapsto\  u_0^*u\in U(p\M p).\ee The above allows us to define the one-to-one maps \be \varphi_{\tilde p p}:\U(\M)\cap\Omega_{\tilde p p}\to(1-\tilde{p})\M\tilde{p}\oplus ip\M^h p\oplus(1-p)\M p\ee of $\U(\M)\cap\Omega_{\tilde p p}$ on the open subset in the real Banach space
$(1-\tilde{p})\M\tilde{p}\oplus ip\M^h p\oplus(1-p)\M p$. Summing up the above facts we can formulate the following statement
\begin{thm}\label{real}
\ben[(i)]
\item The groupoid $\U(\M)$ of partial isometries has a natural structure of the real Banach manifold of the type $\mathfrak{G}$, where the family $\mathfrak{G}$ consist of the real Banach spaces $$(1-\tilde{p})\M\tilde{p}\oplus\tilde{p}\M^h p\oplus(1-p)\M p$$ parameterized  by the pairs $(\tilde p, p)\in \mathcal{L}(\M)\times\mathcal{L}(\M)$ of equivalent projections.
    \item The groupoid $\U(\M)$ is a closed real Banach submanifold of $\G(\M)$ with the real Banach manifold structure underlaying its complex Banach manifold structure.\een\end{thm}
        We conclude from the Proposition \ref{real} the $\G(\M)$ is the complexification of $\U(\M)$ in sense of the definition given in \cite{bou}.

\section {Groupoids and Banach Lie-Poisson structure of $\M_*$}

In \cite{OR} it  was shown that the predual space $\M_*$ of $W^*$-algebra $\M$
has canonically defined  Lie-Poisson structure. This follows from
$ad^*(\M)$-invariance of Banach subspace $\M_*\subset\M^*$, where
$ad_x(y):=xy-yx$. One  defines the Lie-Poisson bracket of $f,g\in
C^\infty(\M_*,\mathbb{C})$ as follows \be\label{Poisbr}
\{f,g\}:=\langle\omega,[Df(\omega), Dg(\omega)]\rangle.\ee Note that Frech\'{e}t
derivatives $Df(\omega), Dg(\omega)$ belong to $\M$ which allows to take the
commutator of them. The predual space $\M_*$ as well as the Lie-Poisson
bracket (\ref{Poisbr}) is invariant with respect to the $Ad^*$-action of the
Banach group $G(\M)$.

Multiplying the right hand side of definition (\ref{Poisbr}) by $i=\sqrt{-1}$ one
obtains the Lie-Poisson bracket for real valued functions $f,g\in
C^\infty(\M_*^h, \mathbb{R})$ defined on the hermitian part $\M_*^h$ of
$\M_*$. This follows from the paring between $\M_*^h$ and the real
Banach-Lie algebra $i\M^h$ of anti-hermitian elements of $\M$ defined by
$$\M_*^h\times i\M^h\ \ni\ (\omega,x)\ \mapsto \ i\langle\omega,x\rangle\ \in \ \mathbb{R}.$$
As in the complex case the Banach Lie-Poisson structure of $\M_*^h$ is
$Ad^*(U(\M))$-invariant. The orbits of coadjoint representation of $U(\M)$ are
weak symplectic manifold. Thus they give symplectic foliation of  Banach
Lie-Poisson space $(\M^h_*,\{\cdot,\cdot\})$, see \cite{OR}.
More information concerning this interesting subject one can find in \cite{BR} and \cite{OR}.

Now let us apply the definitions of the groupoids structures on the tangent
bundle $TG$ and  cotangent bundle $T^*G$ of a Lie group $G$, e.g. see
\cite{mac}, to the case of Banach-Lie group $G(\M)$. We will do this with some
modification. Namely in our considerations we replace the cotangent bundle
$T^*G(\M)$ by the pre-cotangent bundle $T_*G(\M)$ of $G(\M)$. Note that in
the finite dimensional case the bundles $T^*G$ and $T_*G$ are canonically
isomorphic. In our case the cotangent bundle $T^*G(\M)$, opposite  to the
pre-cotangent bundle $T_*G(\M)$ does not have the symplectic structure
related to the Banach Lie-Poisson structure of $\M_*^h$ defined by
(\ref{Poisbr}).

The groupoid structure on $TG(\M)$ is defined as follows. The  base of
$TG(\M)$ is the tangent space $T_eG(\M)$ at the identity element $e\in G(\M)$.
The source map $\Ss:TG(\M)\ \to\ T_eG(\M)$ and the target map $\Tt:TG(\M)\
\to\ T_eG(\M)$ are defined as follows
\be\label{gru1}\begin{array}{l} \Ss(a):=DL_{\pi(a)^{-1}}(\pi(a))a,\\
 \\
 \Tt(a):=DR_{\pi(a)^{-1}}(\pi(a))a,\end{array}\ee
 where $a\in TG(\M)$ and $\pi:TG(\M)\to G(\M)$ is the canonical projection on the base. The identity section $\varepsilon:\ T_eG(\M) \ \rightarrow\ TG(\M)$ is done by the inclusion of the fibre $T_eG(\M)\subset TG(\M)$. The involution $\iota:\ TG(\M)\ \rightarrow\ TG(\M)$ one defines by
 \be\label{gru2} \iota(a):=DL_{\pi(a)^{-1}}(e)\circ DR_{\pi(a)^{-1}}(\pi(a))a.\ee
 Finally the groupoid product is defined by
 \be\label{gru3} ab:=DL_{\pi(a)}(\pi(b))b\ee
 if and only if $(a,b)\in TG(\M)^{(2)}$, i.e. $\Ss(a)=\Tt(b)$.

 As a base for groupoid structure of $T_*G(\M)$ we
 assume the pre-cotangent bundle $T_{*e}G(\M)$ at $e\in G(\M)$.
 The identity section  $\varepsilon_*:\  T_{*e}G(\M) \rightarrow\ T_*G(\M)$
 we define as an inclusion $T_{*e}G(\M)\subset T_{*}G(\M)$.

 Let us take $\xi\in T_{*}G(\M)$ and let $\pi_*(\xi)\in G(\M)$ be the projection of $\xi$ on the base.
 Then one defines the source and target maps as follows:
 \be\label{gru*1}\begin{array}{l}\Ss_*(\xi):=(DL_{\pi(\xi)}(e))^*\xi,\\
 \\
 \Tt_*(\xi):=(DR_{(\pi(\xi))^{}}(e))^*\xi,\end{array}.\ee
 The inversion  $\iota_*:\ T_*G(\M)\ \rightarrow\ T_*G(\M)$ is
defined by \be\label{gru*2}\iota_*(\xi):=(DL_{\pi(\xi)}(\pi(\xi))^{-1})^*\circ
(DR_{(\pi(\xi))^{}}(e))^*\xi.\ee The product of elements $\xi,\eta\in T_{*}G(\M)$ such
that $\Ss(\xi)=\Tt(\eta)$ is given by \be\label{gru*3}
\xi\eta:=(DL_{(\pi(\xi))^{-1}}(\pi(\xi)\pi(\eta)))^*\eta.\ee

The precotangent bundle $T_*G(\M)$ is the weak symplectic complex Banach manifold with the weak symplectic form defined in the following way
\be\label{symplform}
\Omega_L(g,\rho)((a, \xi), (b,\eta))=
\ee
$$=\langle \eta, DL_{g^{-1}}(g)a\rangle+\langle \xi, DL_{g^{-1}}(g)b\rangle-\langle \rho, [DL_{g^{-1}}(g)a,DL_{g^{-1}}(g)b]\rangle,$$
where $g\in G(\M)$, $a,b\in T_gG(\M)$, $\rho, \xi,\eta\in T_{*e}G(\M)$, see \cite{OR}. Thus defined weak symplectic structure is consistent with the groupoid structure of  $T_*G(\M)$ in the sense of \cite{kara}, \cite{weinstein2}. Hence one can consider $T_*G(\M)$ as a weak symplectic groupoid.

The  definition of  action groupoid  structure on the product  $G\times M$,
where $G$ is a group acting on a set $M$, one finds in Appendix D. From this
general definition one gets action groupoid structures on $G(\M)\times \M$
and $G(\M)\times\M_*$ defined by adjoint $Ad:G(\M)\to Aut \ \M$ and
co-adjoint $Ad^*:G(\M)\to Aut \ \M_*$ representation of Banach-Lie group
$G(\M)$: \be Ad_gx=gxg^{-1}\ee \be \langle
Ad^*_g\omega,x\rangle:=\langle\omega, Ad_{g^{-1}}x\rangle,\ee where $x\in
\M$ and $\omega\in \M_*$ respectively.

 The vector bundles  trivializations $\phi:\ TG(\M)\to
G(\M)\times \M$ and $\phi_*:\ T_*G(\M)\to G(\M)\times\M_*$ defined by
\be\label{fi} \phi(a):=(\pi(a),DL_{\pi(a)^{-1}}(\pi(a))a)\ee \be\label{fi*}
\phi_*(\xi):=(\pi(\xi),(DL_{\pi(\xi)}(e))^*\xi)\ee
give the canonical groupoid isomorphisms
$\phi:TG(\M)\to G(\M)\times\M$ and $\phi_*:T_*G(\M)\to G(\M)\times\M_*$. To this
end we take the identifications $T_eU(\M)\cong \M$ and $T_{*e}G(\M)\cong
\M_*$.

Let us define the  injective immersions of the groupoids $\Lambda:TG(\M)  \to\
\G(\M)*_l\M$ and
 {$\Lambda_*:T_*G(\M)  \to\ \G(\M)*_l\M_*$} by:
 \be\label{lambda} \Lambda(a):=\left(\pi(a)\ l(DL_{\pi(a)^{-1}}(\pi(a))a),DL_{\pi(a)^{-1}}(\pi(a))a\right)\ee
 \be\label{lambda*}  \Lambda_*(\xi):=\left(\pi(\xi)\ l((DL_{\pi(\xi)}(e))^*\xi),DL_{\pi(\xi)}(e)^*\xi\right)\ee
 respectively.

 In order to see that $\Lambda:TG(\M)  \to\
\G(\M)*_l\M$ commutes with source and target maps we notice that

$$(\tilde{\Ss}\circ\Lambda )(a)=\tilde{\Ss}\left(\pi(a)\ l(DL_{\pi(a)^{-1}}(\pi(a))a),DL_{\pi(a)^{-1}}(\pi(a))a\right)=\ \ \ \ \ \ \ \ \ \ \ \ \ \ \ \ \ \ \ \ $$
$$  \ \ \ \ \ \ \ \ \ \ \ \ \ \  \ \ \ \ \ \ \ \ \ \ \ \ \ \ \ \ \ \ =DL_{\pi(a)^{-1}}(\pi(a))a=(id\circ\Ss )(a),$$
$$(\tilde{\Tt}\circ \Lambda)(a)=\tilde{\Tt}\left(\pi(a)\ l(DL_{\pi(a)^{-1}}(\pi(a))a),DL_{\pi(a)^{-1}}(\pi(a))a\right)=\ \ \ \ \ \ \ \ \ \ \ \ \ \ \ \ \ \ $$
$$\ \ \ \ \ \ \ \ =Ad_{\pi(a)\ l(DL_{\pi(a)^{-1}}(\pi(a))a)}DL_{\pi(a)^{-1}}(\pi(a))a=Ad_{\pi(a)}DL_{\pi(a)^{-1}}(\pi(a))a=(id\circ\Tt)(a).$$

Since $l(DL_{\pi(a)^{-1}}(\pi(a))a)=Ad_{\pi(a)}l(DL_{\pi(b)^{-1}}(\pi(b))b)$ the following shows that $\Lambda$ preserves also the groupoid product
$$\Lambda(a)\Lambda(b)=\ \ \ \  \ \ \ \ \ \ \ \ \ \  \ \ \ \ \ \ \ \ \ \ \ \ \ \ \ \ \ \ \ \ \ \  \ \ \ \ \ \ \ \ \ \ \ \ \ \ \ \ \ $$ $$=\left(\pi(a)\ l(DL_{\pi(a)^{-1}}(\pi(a))a),DL_{\pi(a)^{-1}}(\pi(a))a\right)
\left(\pi(b)\ l(DL_{\pi(b)^{-1}}(\pi(b))b),DL_{\pi(b)^{-1}}(\pi(b))b\right)=$$
$$=\left((\pi(a)\ l(DL_{\pi(a)^{-1}}(\pi(a))a)(\pi(b))\ l(DL_{\pi(b)^{-1}}(\pi(b))b),DL_{\pi(b)^{-1}}(\pi(b))b\right)=$$
$$ \ \ \ \ \ \ \ \ \ \  \ \ \ \ \ \ \ \ \ \ =\Lambda(DL_{\pi(a)}(\pi(b))b)=\Lambda(ab).$$

Summarizing the  above facts we obtain the following groupoid monomorphism
 \be\label{duzy}\begin{picture}(11,4.6)
\put(1,4){\makebox(0,0){$TG(\M)$}}
    \put(8,4){\makebox(0,0){$\G(\M)*_l\M$}}
    \put(1,-1){\makebox(0,0){$\M$}}
    \put(8,-1){\makebox(0,0){$\M$}}
    \put(1.2,3){\vector(0,-1){3}}
    \put(0.7,3){\vector(0,-1){3}}
    \put(8.2,3){\vector(0,-1){3}}
    \put(7.7,3){\vector(0,-1){3}}
    \put(2.5,4){\vector(1,0){3.7}}
    \put(2,-1){\vector(1,0){4.9}}
    \put(-0.2,1.4){\makebox(0,0){$\Ss$}}
    \put(2.2,1.4){\makebox(0,0){$\Tt$}}
    \put(9.2,1.4){\makebox(0,0){$\tilde{\Tt}$}}
    \put(6.8,1.4){\makebox(0,0){$\tilde{\Ss}$}}
    \put(4.5,4.6){\makebox(0,0){$\Lambda$}}
    \put(4.5,-0.5){\makebox(0,0){$id$}}
      \end{picture}.\ee

\bigskip

In a similar way we obtain the monomorphism
\be\label{duzy*}\begin{picture}(11,4.6)
    \put(1,4){\makebox(0,0){$T_*G(\M)$}}
    \put(8,4){\makebox(0,0){$\G(\M)*_l\M_*$}}
    \put(1,-1){\makebox(0,0){$\M_*$}}
    \put(8,-1){\makebox(0,0){$\M_*$}}
    \put(1.2,3){\vector(0,-1){3}}
    \put(0.7,3){\vector(0,-1){3}}
    \put(8.2,3){\vector(0,-1){3}}
    \put(7.7,3){\vector(0,-1){3}}
    \put(2.5,4){\vector(1,0){3.5}}
    \put(2,-1){\vector(1,0){4.9}}
    \put(-0.2,1.4){\makebox(0,0){$\Ss$}}
    \put(2.2,1.4){\makebox(0,0){$\Tt$}}
    \put(9.2,1.4){\makebox(0,0){$\tilde{\Tt}$}}
    \put(6.8,1.4){\makebox(0,0){$\tilde{\Ss}$}}
    \put(4.5,4.5){\makebox(0,0){$\Lambda_*$}}
    \put(4.5,-0.5){\makebox(0,0){$id$}}
    \end{picture}\ee
    \bigskip
    \\
    of the predual groupoids $T_*\G(\M)\rightrightarrows\M_*$ into the action groupoid $\G(\M)*_l\M_*\rightrightarrows\M_*$

Now instead of complex Banach-Lie group $G(\M)$ let us consider  the groupoid of partially invertible elements $\G(\M)$. In this case we come to the following statements.

The tangent prolongation $T\G(\M)\too T\mathcal{L}(\M)$ of the groupoid $\G(\M)\too\mathcal{L}(\M)$ is a Banach Lie $\mathcal{VB}$-groupoid, (see e.g. \cite{mac2}), i.e. one has
\unitlength=5mm
\be\label{VB}\begin{picture}(11,4.6)
    \put(1,4){\makebox(0,0){$T\G(\M)$}}
    \put(8,4){\makebox(0,0){$\G(\M)$}}
    \put(1,-1){\makebox(0,0){$T\mathcal{L}(\M)$}}
    \put(8,-1){\makebox(0,0){$\mathcal{L}(\M)$}}
    \put(1.2,3){\vector(0,-1){3}}
    \put(0.7,3){\vector(0,-1){3}}
    \put(8.2,3){\vector(0,-1){3}}
    \put(7.7,3){\vector(0,-1){3}}
    \put(2.5,4){\vector(1,0){4}}
    \put(2.2,-1){\vector(1,0){4.5}}
    \put(-0.2,1.4){\makebox(0,0){$D\Ss$}}
    \put(2.2,1.4){\makebox(0,0){$D\Tt$}}
    \put(9.2,1.4){\makebox(0,0){$\Tt$}}
    \put(6.8,1.4){\makebox(0,0){$\Ss$}}
    \put(4.5,4.6){\makebox(0,0){$\tilde{q}$}}
    \put(4.5,-0.5){\makebox(0,0){$q$}}
    \end{picture}\ee
    \bigskip
    \\
where the bundle vector projections $q$ and  $\tilde{q}$ define the groupoid morphism, the tangent maps $D\Ss, D\Tt, D\iota, D\varepsilon $  are vector bundle morphisms, and the map $$(Dq,D\Ss):T\G(\M)\to \G(\M)*T\mathcal{L}(\M)$$  of tangent groupoid $T\G(\M)$ on the action groupoid $\G(\M)*T\mathcal{L}(\M)$ is a surjective submersion.

The core of $T\G(\M)$ is the algebroid $\A\G(\M)$ of the groupoid $\G(\M)$. See e.g.  \cite{mac2} for the definition of the core of $\mathcal{VB}$-groupoid. The algebroid $\A\G(\M)$ and its predual $\A_*\G(\M)$  are most crucial for the Poisson aspect of the investigated theory. Namely, extending the considerations from the finite dimensional case to the Banach-Lie context we obtain the Banach-Lie  $\mathcal{VB}$-groupoid

\unitlength=5mm
\be\label{VB*}\begin{picture}(11,4.6)
    \put(1,4){\makebox(0,0){$T_*\G(\M)$}}
    \put(8,4){\makebox(0,0){$\G(\M)$}}
    \put(1,-1){\makebox(0,0){$\A_*(\M)$}}
    \put(8,-1){\makebox(0,0){$\mathcal{L}(\M)$}}
    \put(1.2,3){\vector(0,-1){3}}
    \put(0.7,3){\vector(0,-1){3}}
    \put(8.2,3){\vector(0,-1){3}}
    \put(7.7,3){\vector(0,-1){3}}
    \put(2.5,4){\vector(1,0){4}}
    \put(2,-1){\vector(1,0){4.7}}
    \put(-0.2,1.4){\makebox(0,0){$\tilde{\Ss}$}}
    \put(2.2,1.4){\makebox(0,0){$\tilde{\Tt}$}}
    \put(9.2,1.4){\makebox(0,0){$\Tt$}}
    \put(6.8,1.4){\makebox(0,0){$\Ss$}}
    \put(4.5,4.6){\makebox(0,0){$\tilde{q}_*$}}
    \put(4.5,-0.5){\makebox(0,0){$q_*$}}
    \end{picture}\ee
    \bigskip
    \\
precotangent to the one presented in (\ref{VB}), where $q_*$ and $\tilde{q}_*$ are the projections on  the base. One defines the source $\tilde\Ss$ and target $\tilde\Tt$ maps in (\ref{VB*}) as follows. Let $\phi\in T_{*x}\G(\M)$ and $x\in \A_p\G(\M)$, $p\in \mathcal{L}(\M)$, then

\be \langle \tilde\Ss(\phi),x\rangle:=\langle\phi, DL_g(\varepsilon(p))(x-D\varepsilon(p)D\Tt(\varepsilon(p))x)\rangle,
\ee
\be\langle\tilde\Tt(\phi)\rangle:=\langle\phi,DR_g(\varepsilon(p))x\rangle.
\ee
The product $\phi\bullet\psi$ of $\phi\in T_{*x}\G(\M)$ and $\psi\in T_{*y}\G(\M)$, where $\tilde\Ss(\phi)=\tilde\Tt(\psi)\in\A_{*p}\G(\M)$ and $\Ss(x)=\Tt(y)=p\in\mathcal{L}(\M)$, one defines by

\be\langle\phi\bullet\psi, \xi\cdot\eta\rangle=\langle\phi, \xi\rangle+\langle\psi,\eta\rangle,\ee
where $\xi\in T_x\G(\M)$, $\eta\in T_y\G(\M)$ satisfy $D\Ss(\xi)=D\Tt(\eta)$ and $\xi\cdot\eta\in T_{xy}\G(\M)$ is the product of $\xi$ and $\eta$ in the tangent groupoid $T\G(\M)$. The above definitions we obtain as a direct generalization of those accepted in the finite dimensional case, e.g. \cite{mac2}.

The groupoid $T_*\G(\M)\rightrightarrows\A_*\G(\M)$ is a weak symplectic Banach-Lie realization of the Banach-Poisson bundle $ \A_*\G(\M)$, which Poisson structure is determined by the algebroid structure of $\A\G(\M)$. We note here that diagram (\ref{VB*}) is the groupoid version of the diagram
\unitlength=5mm
\be\label{VBgr}\begin{picture}(11,4.6)
    \put(1,4){\makebox(0,0){$T_*G(\M)$}}
    \put(8,4){\makebox(0,0){$G(\M)$}}
    \put(1,-1){\makebox(0,0){$\M_*$}}
    \put(8,-1){\makebox(0,0){$\{\mathbf{1}\}$}}
    \put(1.2,3){\vector(0,-1){3}}
    \put(0.7,3){\vector(0,-1){3}}
    \put(8.2,3){\vector(0,-1){3}}
    \put(7.7,3){\vector(0,-1){3}}
    \put(2.5,4){\vector(1,0){4}}
    \put(2,-1){\vector(1,0){4.7}}
    \put(-0.2,1.4){\makebox(0,0){$\tilde{\Ss}$}}
    \put(2.2,1.4){\makebox(0,0){$\tilde{\Tt}$}}
    \put(9.2,1.4){\makebox(0,0){$\Tt$}}
    \put(6.8,1.4){\makebox(0,0){$\Ss$}}
    \put(4.5,4.6){\makebox(0,0){$\tilde{q_*}$}}
    \put(4.5,-0.5){\makebox(0,0){$q_*$}}
    \end{picture}\ee
    \bigskip
    \\
    valid for the group $G(\M)$.

The proofs of these statements are the direct generalizations of the proofs for the finite dimensional case to the context of the Banach-Lie groupoids theory.

Finally let us mention that all objects considered above belong to the category of complex analytic Banach manifold. They have their real analytic counterparts if we replace  the group $G(\M)$ and the groupoid $\G(\M)$ by $U(\M)$ and $\U(\M)$ respectively, and $\M$ ($\M_*$) by $i\M^h$ ($\M_*^h$).

\bigskip

Authors apologize for the fact that subjects  discussed in this section are treated in the abbreviated way. However, the detailed investigation of Banach -Lie Poisson geometry related to $W^*$-algebras needs longer treatment in a separate paper, which is currently in preparation.

\section {Appendix}

{\large\textbf{A \ Groupoid \ }}\\
 Let us recall that a \textbf{groupoid with the base set $B$} (set of objects) is a set $G$ such that:
\ben[i)]
\item
there is a pair of maps $$\bfig \Atriangle/>`>`/[G` {B}` {B};\operatorname{\Ss}`\operatorname{\Tt}`] \efig $$  called {\bf source} and {\bf target} map respectively;
\item
for set of composable pairs  $$G^{(2)}:=\{(g,h)\in G\times G; \ \ \ {\Ss}(g)={\Tt}(h)\}$$ one has a  \textbf{product map} $m:G^{(2)}\rightarrow G$, denoted by
\be\label{prod-map}m(g,h)=:gh\ee such that
\ben[(a)]
\item ${\Ss}(gh)={\Ss}(h),\ \ \ {\Tt}(gh)={\Tt}(g),$
\item associativity: $k(gh)=(kg)h;$
\een
\item
there is an  iniection  $\varepsilon:{B}\rightarrow G$ called the \textbf{identity section}, such that $$\varepsilon({\Tt}(g))g=g=g\varepsilon({\Ss}(g));$$
\item
there exists an \textbf{inversion} $\iota:G\rightarrow  G$ denoted by
\be\label{inv-map}\iota(g)=:g^{-1},\ee
 such that $$\iota(g)g=\varepsilon({\Ss}(g)),\ \ \ \ \  g\iota(g)=\varepsilon ({\Tt}(g))$$
for all  $g\in G.$
\een
A groupoid $G$ gives rise to a hierarchy of sets\\
$G^{(0)}:=\varepsilon({B})\simeq ({B})$\\
$G^{(1)}:=G$\\
$G^{(2)}:=\{(g,h)\in G\times G; \ \ \ {\Ss}(g)={\Tt}(h)\}$\\
$\vdots$\\
$G^{(k)}=\{(g_1,g_2,...,g_k)\in G\times G\times...\times G;\ \ {\Tt}(g_i)={\Ss}(g_{i-1}),\ \ i=2,3,...,k\}$

In the paper we will consider the topological (differentiable) groupoids.
Because of this let us recall that the groupoid $G$ is called a topological (differentiable) groupoid if $G$ and $B$ have the topologies (differential manifold structure) such that: \ben[i)]
\item
the product map (\ref{prod-map}) and the involution (\ref{inv-map}) are
continuous (differentiable);
\item
the injection $\varepsilon:{B}\rightarrow G$ is an embedding (differentiable
embedding). \een

From ${\Ss}(g)=\varepsilon^{-1}(gg^{-1})$ and
${\Tt}(g)=\varepsilon^{-1}(g^{-1}g)$ it follows that source map and target map
are continuous (differentiable). By definition the topology of $G^{(k)}$, for
$k=0,1,2,...$, is inherited from $G$. In case of differentiable groupoid one
assumes that the source and target maps are submersions.\\
\\
{\large\textbf{B \ Groupoids morphism\ }}\\
A \textbf{morphism} $\phi$ of two groupoids $G_1$ and $G_2$ over bases $B_1$ and $B_2$ one can depict by the following commutative diagram
\unitlength=5mm
\be\label{morph}\begin{picture}(11,4.6)
    \put(1,4){\makebox(0,0){$G_1$}}
    \put(8,4){\makebox(0,0){$G_2$}}
    \put(1,-1){\makebox(0,0){$B_1$}}
    \put(8,-1){\makebox(0,0){$B_2$}}
    \put(1.2,3){\vector(0,-1){3}}
    \put(0.7,3){\vector(0,-1){3}}
    \put(8.2,3){\vector(0,-1){3}}
    \put(7.7,3){\vector(0,-1){3}}
    \put(2.5,4){\vector(1,0){4}}
    \put(2,-1){\vector(1,0){4.7}}
    \put(-0.2,1.4){\makebox(0,0){$\Ss_1$}}
    \put(2.2,1.4){\makebox(0,0){$\Tt_1$}}
    \put(9.2,1.4){\makebox(0,0){$\Tt_2$}}
    \put(6.8,1.4){\makebox(0,0){$\Ss_2$}}
    \put(4.5,4.6){\makebox(0,0){$\phi_G$}}
    \put(4.5,-0.5){\makebox(0,0){$\phi_B$}}
    \end{picture}\ee
\bigskip
\\
By definition one also has
$$\phi_B\circ\Ss_1=\Ss_2\circ\phi_G\ \ \ {\rm\ and}\ \ \phi_B\circ\Tt_1=\Tt_2\circ\phi_G$$
and $$\phi_G(g)\phi_G(h)=\phi_G(gh)$$ for $gh\in G_1^{(2)}$.
If $\phi_G:G_1\hookrightarrow G_2$ and $\phi_B:B_1\hookrightarrow B_2$ are inclusion maps one says that $G_1$ is \textbf{subgroupoid} of $G_2$. The subgroupoid $G_1\subset G_2$ is a \textbf{wide subgroupoid} of $G_2$ if $\Ss_1(G_1)=\Tt_1(G_1)=B_2$.

An example of groupoid morphism is given by
\unitlength=5mm
\be\begin{picture}(11,4.6)
    \put(1,4){\makebox(0,0){$G$}}
    \put(8,4){\makebox(0,0){$B\times B$}}
    \put(1,-1){\makebox(0,0){$B$}}
    \put(8,-1){\makebox(0,0){$B$}}
    \put(1.2,3){\vector(0,-1){3}}
    \put(0.7,3){\vector(0,-1){3}}
    \put(8.2,3){\vector(0,-1){3}}
    \put(7.7,3){\vector(0,-1){3}}
    \put(2.5,4){\vector(1,0){4}}
    \put(2,-1){\vector(1,0){4.7}}
    \put(-0.2,1.4){\makebox(0,0){$\Ss$}}
    \put(2.2,1.4){\makebox(0,0){$\Tt$}}
    \put(9.2,1.4){\makebox(0,0){$pr_2$}}
    \put(6.8,1.4){\makebox(0,0){$pr_1$}}
    \put(4.5,4.6){\makebox(0,0){$(\Ss,\Tt)$}}
    \put(4.5,-0.5){\makebox(0,0){$id$}}
    \end{picture}\ee
\bigskip
\\
where $B\times B$ is the pair groupoid, i.e. $\Ss:=pr_1$, $\ \Tt:=pr_2$, $\ \iota(x,y):=(y,x)$, $\ \varepsilon(x)=(x,x)$ and $\ m((y,z),(x,y))=(x,z)$.\\
 \\
{\large\textbf{C \ Action of groupoid \ }}\\
We recall the definition of the \textbf{ left action of groupoid $G$ on the set $M$}. One assumes for this reason that there exists a map (moment map)
\be \mu: M \rightarrow B\ee
and one defines the space
\be G*_l M:=\{(g,r)\in G\times M:\ \ \ \Ss(g)=\mu (r)\}.\ee
Then the left action of groupoid $G$  on M is defined as a map $G*M \ni (g,r)\mapsto g\cdot r\in M$ with properties:
\be\label{agl}
\begin{array}{l}
(gh)\cdot r=g\cdot(h\cdot r)\\
\mu(g\cdot r)={\Tt}(g)\\
\varepsilon(\mu(r))\cdot r=r.\end{array}\ee
For the \textbf{right action of $G$ on $M$} instead of (\ref{agl}) we have
\be\label{agr}
\begin{array}{l}
 r\cdot(gh)=(r\cdot h)\cdot g\\
 \mu(r\cdot g)={\Ss}(g)\\
 r\cdot\varepsilon(\mu(r))=r,\end{array}\ee
where $(g,r)\in G*_r M:=\{(g,r)\in G\times M:\ \ \ \Tt(g)=\mu (r)\}.$

As an example let us take the canonical left action of $G$ on its base $B$. In this case $M:=B$, $\ \mu:=id$ and \be G*B=\{(g,x):\ \ x={\Ss}(g)\}\ee
The action map is defined by
\be G*B\ \ni\ (g,x)\ \mapsto\ g\cdot x:={\Tt}(g).\ee
The defining properties (\ref{agl}) follow from the corresponding properties of the maps $\Ss,\Tt,\ \varepsilon$ and the product map (\ref{prod-map}).

One equips the set   $\tilde G:=G*M$ with the groupoid structure defined as follows:
\ben[i)]
\item
source map and target map are given by $\tilde \Ss(g,r):=r\in M\ $ and $\ \tilde \Tt(g,r):=g\cdot r\ \in M$;
\item
the set of composable pairs $$\tilde G^{(2)}:=\{((g,r),(h,n))\in \tilde G\times \tilde G;\ \ \Tt(h)=\Ss(g)\}$$ and the product map $\tilde m:\tilde G^{(2)}\rightarrow \tilde G$ is defined as
\be\label{*prod}\tilde m((g,r),(h,n))=(gh,n);\ee
\item
the identity section  $\tilde \varepsilon:M\rightarrow \tilde G$ is defined by
\be\label{*id} \tilde \varepsilon(r)=(\varepsilon(\mu(r)),r);\ee
\item
the involution $\tilde \iota:\tilde G\rightarrow\tilde G$  is defined by
\be\label{*inv} \tilde\iota(g,r)=(\iota(g),g\cdot r).\ee
\een
%The proof:\\
%$\tilde s(\tilde m((g,r),(h,n)))=\tilde s((n,gh))=n=\tilde s
%(h,n)$\\
%$\tilde t(\tilde m((g,r),(h,n)))=\tilde t((n,gh))=gh\cdot n=g(h\cdot
%n)=g\cdot r=\tilde t(g,r)$\\
%the associativity of $\tilde m$ follows from the associativity of
%$m$\\
%$\tilde \varepsilon(\tilde t(g,r))\cdot(g,r)= \tilde
%\varepsilon(g\cdot r)(g,r) =(\varepsilon(\mu(g\cdot
%r)),r)\cdot(g,r)=(\varepsilon(t(g)),r)\cdot(g,r)=(\varepsilon(t(g))\cdot
%g,r)=(g,r)$\\
%$(g,r)\cdot \tilde \varepsilon(\tilde s(g,r))=(g,r)\cdot\tilde
%\varepsilon(r)=(g,r)\cdot (\varepsilon(\mu(r)),r)=(g,r)\cdot
%(\varepsilon(s(g)),r)=(g\varepsilon(s(g)),r)=(g,r)$\\
%$\tilde\iota(g,r)(g,r)=(\iota(g),g\cdot
%r)(g,r)=9\iota(g)g,r)=(\varepsilon(s(g)),r)=(\varepsilon(\mu(r)),r)=\tilde
%\varepsilon(r)=\tilde \varepsilon(\tilde s(g,r))$\\
%$(g,r)\tilde\iota(g,r)=(g,r)(\iota(g),g\cdot r)=(g\iota(g),g\cdot
%r)=(\varepsilon(t(g)),g\cdot r)=(\varepsilon(\mu(g\cdot r)),g\cdot
%r)=\tilde \varepsilon(g\cdot r)=\tilde \varepsilon(\tilde t(g,r)).$
In the case when $G$ is a topological groupoid and $M$ is a topological space we obtain on $\tilde G$ the structure of the topological groupoid if the moment map $\mu$ and the action $G$ on $M$ are continuous. The topological structure of $\tilde G\subset G\times M$
is inherited from product topology of $G\times M$.

One calls the morphism depicted in (\ref{morph}) a \textbf{covering morphism} if for each $x\in B_1$ the restriction $\phi_G:\Ss^{-1}(x)\ \rightarrow\ \Ss^{-1}(\phi_B(x))$ of $\phi_G$ to the $\Ss$-level of $x$ is bijection.

The diagram
\unitlength=5mm
\be\label{cov morph}\begin{picture}(11,4.6)
    \put(1,4){\makebox(0,0){$G*_l M$}}
    \put(8,4){\makebox(0,0){$G$}}
    \put(1,-1){\makebox(0,0){$M$}}
    \put(8,-1){\makebox(0,0){$B$}}
    \put(1.2,3){\vector(0,-1){3}}
    \put(0.7,3){\vector(0,-1){3}}
    \put(8.2,3){\vector(0,-1){3}}
    \put(7.7,3){\vector(0,-1){3}}
    \put(2.5,4){\vector(1,0){4}}
    \put(2,-1){\vector(1,0){4.7}}
    \put(-0.2,1.4){\makebox(0,0){$\tilde \Ss$}}
    \put(2.2,1.4){\makebox(0,0){$\tilde \Tt$}}
    \put(9.2,1.4){\makebox(0,0){$\Tt$}}
    \put(6.8,1.4){\makebox(0,0){$\Ss$}}
    \put(4.5,4.6){\makebox(0,0){$pr_1$}}
    \put(4.5,-0.5){\makebox(0,0){$\mu$}}
    \end{picture}\ee
\bigskip
\\
where $$\phi_G(g,r):=pr_1(g,r)=g \ \ \ {\rm and}\ \ \ \phi_B(r):=\mu(r), $$ gives an example of the covering morphism of groupoids.\\
\\
{\large\textbf{D \ Action groupoid \ }}\\
If a group $G$ acts on a set $M$
$$G\times M\  \ni\ (g,m)\  \mapsto\ g\cdot m\ \in\ M\ $$
 one can define on the set $G\times M$ the groupoid structure, which is called \textbf{action groupoid} structure.
For this case one defines \ben[i)]
\item
source  and target maps $\Ss', \ \Tt':G\times M\to M$ as
\be\label{'st}\Ss'(g,m):=m\in M\ \quad{\rm and }\quad  \Tt(g,m):=g\cdot m\ ;\ee
\item the groupoid product \be\label{'prod}(g,m)(h,n):=(gh,n)\ee
on the set of composable pairs $$(G\times M) ^{(2)}:=\{((g,m),(h,n))\in (G\times M)\times (G\times M):\ \ m=h\cdot n\}$$
\item
the identity section  $\varepsilon':M\rightarrow G\times M$ by
\be\label{'id}\varepsilon'(m)=(e,m);\ee
\item
the involution $\iota':G\times M \rightarrow G\times M$   by \be\label{'inv}
\iota'(g,m)=(g^{-1},g\cdot m).\ee \een

{\large\textbf{E\ Bisections of groupoid\ }}\\
By a \textbf{bisection } of a groupoid $G$ one understands a subset $\sigma$  such that the restrictions $\Ss:\sigma\to {B}$ and  $\Tt:\sigma\to {B}$ of the source map  and target map to $\sigma$ are bijections of $\sigma$ on $B$. The set $\mathfrak{B}(G)$ of all bisections of the groupoid $G$ form  the group if one defines the product of two bisections $\sigma_1$ and $\sigma_2$ as follows:
\be
\sigma_1\circ \sigma_2:=\{gh\in G;\ \ (g,h)\in (\sigma_1\times \sigma_2)\ \cap\ G^{(2)}\}
\ee
The identity element of the \textbf{bisection group} $(\mathfrak{B}(G),\circ)$ is just the identity section ${B}\cong G^{(0)}$. The map $$\mathfrak{B}(G)\ \ni\ \sigma\ \to\ (\Tt|_\sigma)\circ(\Ss|_\sigma)^{-1}\ \in\ Bij\ {B}$$ is the group monomorphism. Therefore one can consider $\mathfrak{B}(G)$ as a subgroup of the group $Bij\ {B}$ of bijections of ${B}$.

A \textbf{local bisection} of a groupoid $G$ is a subset $\sigma\subset G$ such that $\Ss:\sigma\to {B}$ and $\Tt:\sigma\to {B}$ are injections on $\sigma$ into ${B}$. The maps
\be \Tt\circ \Ss^{-1}:\ \Ss(\sigma)\to \Tt(\sigma) \;\;{\rm and}  \;\;\Ss\circ \Tt^{-1}:\ \Tt(\sigma)\to \Ss(\sigma)\ee
are \textbf{partial bijections} of ${B}$. They define the inverse subsemigroup $\mathfrak{B}_{loc}(G)\subset B_{part}({B})$ of the inverse semigroup $B_{part}({B})$ of all partial bijections of ${B}$. One calls $\mathfrak{B}_{loc}(G)$ the inverse semigroup of local bisections.  Let us recall that the semigroup product in $B_{part}({B})$ is just the superposition of partial bijections.

Let us remark a motivation to use terms "bisection" and "local bisection" is that $(\Ss|_\sigma)^{-1}:{B}\to \sigma\subset G$ ($(\Tt|_\sigma)^{-1}:{B}\to \sigma\subset G$) and $(\Ss|_\sigma)^{-1}:\Ss(\sigma)\to \sigma\subset G$  ($(\Tt|_\sigma)^{-1}:\Tt(\sigma)\to \sigma\subset G$) are section and local section of the bundle $\Ss:G\to{B}$ ($\Tt:G\to {B}$).\\

\begin{center} ACKNOWLEDGEMENTS\end{center}

We thank Kirill Mackenzie for inspiring discussions  within the Lie $\mathcal{VB}$-groupoids and Lie algebroids. We also thank Daniel Beltita for useful remarks which made the paper transparent.

\end{document}